\documentclass[12pt, reqno]{amsart}
\usepackage{amsmath, amsthm, amscd, amsfonts, amssymb, graphicx, color}
\usepackage[bookmarksnumbered, colorlinks, plainpages]{hyperref}
\usepackage{graphicx} 
\graphicspath{ C:\Users\pqdee\OneDrive\Desktop\bf}
\graphicspath{ C:\Users\pqdee\OneDrive\Desktop}
\usepackage{caption}
\usepackage{subcaption}
\usepackage[export]{adjustbox}
\usepackage{wrapfig}
\usepackage{float}
\usepackage{tikz}
\hypersetup{colorlinks=true,linkcolor=red, anchorcolor=green, citecolor=cyan, urlcolor=red, filecolor=magenta, pdftoolbar=true}
\graphicspath{ {./images/} }
\usepackage{wrapfig}
\usepackage{multicol}
\newtheorem{theorem}{Theorem}[section]
\newtheorem{lemma}[theorem]{Lemma}

\newtheorem{corollary}[theorem]{Corollary}
\theoremstyle{definition}

\theoremstyle{remark}
\newtheorem{remark}[theorem]{\bf{Remark}}
\numberwithin{equation}{section}
\newcommand{\vertiii}[1]{{\left\vert\kern-0.25ex\left\vert\kern-0.25ex\left\vert #1 
    \right\vert\kern-0.25ex\right\vert\kern-0.25ex\right\vert}}

\allowdisplaybreaks

\newcommand{\mfn}{\mathbf n}

\newcommand{\D}{\mathbb{D}}

\newcommand{\g}{\gamma}

\newcommand{\p}{\phi}
\renewcommand{\a}{\alpha}
\renewcommand{\o}{\omega}

\begin{document}

\title[Generalized composition operator] {{ Generalized complex symmetric composition operators with applications }}
\author[V. Allu] {Vasudevarao Allu}
\author[S. Sahoo] {Satyajit Sahoo}

\address{Vasudevarao Allu, Department of Mathematics, 
	School of Basic Sciences, Indian Institute of Technology Bhubaneswar, Odisha 752050, India.}
\email{avrao@iitbbs.ac.in}

\address{Satyajit Sahoo, Department of Mathematics, 
	School of Basic Sciences, Indian Institute of Technology Bhubaneswar, Odisha 752050, India.}
\email{satyajitsahoo2010@gmail.com, ssahoomath@gmail.com}

\thanks{
Some part of this work was originally conceived when the second author was a postdoctoral fellow at NISER Bhubaneswar, Odisha, India during 2024. The second author is thankful to NISER Bhubaneswar for providing the necessary facilities to carry out this work and also expresses his gratitude to Prof. Anil Kumar Karn.
}


\date{}
\maketitle
\begin{abstract}
	
We characterize the weighted composition-differentiation operators $D_{\mfn,\psi,\varphi}$ acting on $\mathcal{H}_\gamma(\mathbb{D}^d)$ over the polydisk $\mathbb{D}^d$ which are complex symmetric with respect to the conjugation $\mathcal{J}$.
	We obtain necessary and sufficient conditions for $D_{\mfn,\psi,\varphi}$ to be self-adjoint. We also investigate complex symmetry of generalized weighted composition differentiation  operators $M_{n, \psi, \varphi}=\displaystyle\sum_{j=1}^{n}a_jD_{j,\psi_j, \varphi},$
(where $a_j\in \mathbb{C}$ for $j=1, 2, \dots, n$) on the reproducing kernel Hilbert space $\mathcal{H}_\gamma(\mathbb{D})$ of analytic functions on the unit disk $\mathbb{D}$ with respect to a weighted composition conjugation $C_{\mu, \xi}$. Further, we discuss the structure of self-adjoint linear composition differentiation operators. Finally, the convexity of the Berezin range of composition operator on $\mathcal{H}_\gamma(\mathbb{D})$ are investigated. Additionally, geometrical interpretations have also been employed.
\end{abstract}

%

\tableofcontents

\section{Introduction and motivation}
The general study of complex symmetry has origins that can be traced back to both operator theory and complex analysis, and was initiated by Garcia and Putinar \cite{Garcia_Putinar_2006, Garcia_Putinar_2007}. Thereafter, there has been an extensive research carried out on complex symmetry (see \cite{Garcia_Prodan_Putinar_2014} and the references therein). In the past decade, complex symmetric operators have become particularly important
in both theoretical aspects and applications point of view. Renewed interest in non-Hermitian quantum mechanics and spectral analysis of certain complex symmetric operators has brought forth more research into complex symmetry. The present results show that bounded complex symmetric operators are very diverse, which include the Volterra integration operator, normal operator, compressed Toeplitz operators, etc. Recently, in \cite{Mleczko_Szwedek_2018}, some interpolation properties of complex symmetric operators on Hilbert spaces, with an application to complex symmetric Toeplitz operators, have been studied. \\
\par
A systematic study, at the abstract level, of complex symmetric operators was undertaken in \cite{Garcia_Putinar_2006, Garcia_Putinar_2007, Garcia_Hammond, Hai_PA_2020, Hai_Putinar_JDE_2018}. For applications to mathematical physics	and connections to the theory of operators in an indefinite metric space, we refer to the survey \cite{Garcia_Prodan_Putinar_2014}.
	 Throughout this article $\mathcal{H}$ is a separable complex Hilbert space. An anti-linear operator $\mathcal{C}: \mathcal{H}\rightarrow \mathcal{H}$ is called a {\it conjugation} if it is both involutive ({\it i.e.,} $\mathcal{C}^2=I$ ) and isometric ({\it i.e.,} $\|\mathcal{C}x\|=\|x\|,~ x\in \mathcal{H}$). In the presence of a conjugation $\mathcal{C}$, the inner product of the underlying Hilbert space induces a bounded, complex bilinear symmetric form $\langle x,y \rangle=\langle x, \mathcal{C}y\rangle,~ x, y\in \mathcal{H}$. Symmetry of linear transforms with respect to this bilinear form is the main subject of our study. For any conjugation $\mathcal{C}$, there is an orthonormal basis $\{e_n\}_{n=0}^\infty$ for $\mathcal{H}$ such that $\mathcal{C}e_n=e_n$ for all $n$ (see \cite{Garcia_Putinar_2006} for more details). Let ${\mathcal{L}}(\mathcal{H})$ be the $C^*$-algebra of all bounded linear operators on a complex Hilbert space $\mathcal{H}$ with inner product $\langle \cdot,\cdot\rangle$ and the corresponding norm $\|.\|$.      
     An operator $T\in \mathcal{L(H)}$ is said to be {\it complex symmetric} if there exists a conjugation $\mathcal{C}$ on $\mathcal{H}$ such that $T=\mathcal{C}T^*\mathcal{C}$. In this case, we say that $T$ is {\it complex symmetric} with conjugation $\mathcal{C}$. This concept is due to the fact that $T$ is a complex symmetric operator if, and only if, it is unitarily equivalent to a symmetric matrix with complex entries, regarded as an operator acting on $l^2$-space of the appropriate dimension (see \cite{Garcia_CM_2006}). The class of complex symmetric operators is unexpectedly large. Complex symmetric weighted composition operators on several Hilbert spaces of analytic functions have been studied in \cite{Jung_Kim_Ko_Lee_JFAA_2014, Lim_Khoi_JMAA_2018}. Garcia and Hammond \cite{Garcia_Hammond} have identified new classes of complex symmetric (weighted) composition operators on weighted Hardy spaces. Recently, Han and Wang have considered complex symmetric generalized weighted composition operator on Hardy space $\mathcal{H}^2(\mathbb{D})$ and Bergman space $\mathcal{L}_a^2(\mathbb{D})$ in \cite{Han_Wang_JMAA_2021,Han_Wang_CAOT_2021} respectively. Very recently, Bhuia {\it et al.} \cite{Bhuia_Pradhan_Sarkar_2025} have characterized complex symmetry of a class of composition operators. Further, they constructed a conjugation via a unitary weighted composition operator
and discussed the symmetricity of Toeplitz operators corresponding to composition-based conjugations.

\subsection{One-variable setting}
Let $\mathbb{D}$ be the open unit disk in the complex plane $\mathbb{C}$. A reproducing kernel Hilbert space (RKHS) on $\mathbb{D}$ over $\mathbb{C}$ is a space for which all linear evaluation functionals on $\mathbb{D}$ are bounded. Given any $w\in \mathbb{D}$, it follows from Riesz representation theorem that there exists a unique vector $K_w$ in the RKHS such that $f(w)=\langle f, K_w \rangle$ for every $f$ in the RKHS. We recall $K_w$ the reproducing kernel for the point $w$. For a detailed discussion of a RKHS we refer to \cite{Paulsen_Raghupathi_2016}. Let $\mathcal{H}_\gamma (\gamma\in \mathbb{N})$ be the Hilbert space of holomorphic functions on $\mathbb{D}$ with the reproducing kernel $$ K_w(z)=(1-\overline{w}z)^{-\gamma}.$$ Note that for $\gamma=1~ (\mbox{resp}. ~ \gamma=2)$, $\mathcal{H}_\gamma$ reduces to the classical  Hardy space (resp. Bergman space). In fact, such a space was studied in \cite{Le_JMAA_2012} for several complex variables and for $\gamma>0$. However, we will restrict our study to just $\gamma \in \mathbb{N}$. This restriction $\gamma \in \mathbb{N}$ is essential, as in general, for any $a, b \in \mathbb{C}$ and $\gamma>0$, $(ab)^\gamma\neq a^\gamma b^\gamma$, while the equality holds when $\gamma \in \mathbb{N}$. From the theory of reproducing kernel Hilbert space, $\mathcal{H}_\gamma$ is the completion of the linear span of $\{ K_w: w\in \mathbb{D}\}$ with respect to the inner product defined by $\langle K_w, K_z \rangle=K_w(z)$. For any $z, w\in \mathbb{D}$, the norm of $K_w(z)$ is given by $$\|K_w\|=\langle K_w, K_w \rangle=(1-|w|^2)^{-\frac{\gamma}{2}},$$   
the normalized reproducing kernel  for the point $w$, denoted by $k_w$, is given by $$\widehat{k}_w(z)=\frac{K_w(z)}{\|K_w\|}=\frac{(1-|w|^2)^{\frac{\gamma}{2}}}{(1-\overline{w}z)^{\gamma}}.$$      
Moreover, if we let $$K_{w}^{[n]}(z)=\frac{\Gamma(\gamma+n)}{\Gamma(\gamma)}\frac{z^n}{(1-\overline{w}z)^{\gamma+n}},$$
where $n$ is a non-negative integer and $\Gamma$ is the usual gamma function, then $K_{w}^{[n]}$ is the reproducing kernel for the point evaluation of the $n^{th}$ derivative at $w$; {\it i.e.,}
$$f^{(n)}(w)=\langle f, K_{w}^{[n]} \rangle$$  
for all $f\in \mathcal{H}_\gamma$ (here $f^{(0)}=f$).\\[2mm]
\par
Let $\psi: \mathbb{D}\longrightarrow \mathbb{C}$ be a non-zero analytic function and $\varphi: \mathbb{D}\longrightarrow \mathbb{D}$ be holomorphic functions and $f\in \mathcal{H}_\gamma(\mathbb{D})$. These two functions induce the {\it generalized weighted composition operator} $D_{n, \psi, \varphi}$ on $\mathcal{H}_\gamma$, defined by $$ D_{n, \psi, \varphi}f(z)=\psi(z)f^{(n)}(\varphi(z)),$$
for every $z\in \mathbb{D}$ and $f\in\mathcal{H}_\gamma$ ($D_{n, \psi, \varphi}$ is also called a {\it weighted composition-differentiation  operator}) by other authors in \cite{Fatehi_Hammond_CAOT_2021, Han_Wang_JMAA_2021,Han_Wang_CAOT_2021}. This map generalizes many known operators. For example, if $n=0$ and $\psi \equiv 1$, $D_{n, \psi, \varphi}$ is just the composition operator $C_\varphi$ with symbol $\varphi$. If $n=0$, we obtain the weighted composition operator $\psi C_\varphi$ with the weight $\psi$ and symbol $\varphi$. If $\psi \equiv 1$ and $\varphi(z)=z$,  $ D_{n, \psi, \varphi}$ reduces to $D^n$, the differential operator of order $n$. We observe that $ D_{n, \psi, \varphi}=\psi C_\varphi D^n$. Thus, investigating generalized weighted composition maps provides a unified approach for treating these classes of operators. Operator-theoretic properties of  $D_{n, \psi, \varphi}$ and some of the special cases such as $C_\varphi D^1$ have been studied by several authors (see \cite{Datt_Jain_Ohri_Filomat_2020, Fatehi_Hammond_PAMS_2020, Hibschweiler_Portnoy_Rocky_2005, Ohno_BAMS_2006, Stević_BBMS_2009, Zhu_NFAO_2009}). To avoid triviality, we assume $\psi$ is not identically zero. A simple application of the closed graph theorem shows that $ D_{n, \psi, \varphi}$ maps $\mathcal{H}_\gamma$ into itself if, and only if,  $D_{n, \psi, \varphi}$ is bounded.

\subsection{$n$-variable setting}
Let $\mathbb{C}$ be the complex plane and let $\mathbb{C}^d$ be the $d$-dimensional complex Euclidean space. Let $\mathbb{T}$ denote the boundary of the unit disk $\mathbb{D}=\{z\in \mathbb{C}: |z|<1\}$ in the complex plane $\mathbb{C}$. For a fixed integer $d$, the polydisk $\mathbb{D}^d=\{z=(z_1, \dots, z_d)\in \mathbb{C}^d: |z_j|<1,~1\leq j\leq d\}$ and the torus $\mathbb{T}^d$ can be represented as the Cartesian products of $d$ copies of $\mathbb{D}$ and $\mathbb{T}$, respectively. Although we use the same symbol $z$ to denote a point in $\mathbb{C}$ or a point in $\mathbb{C}^d$, the reader should have no trouble recognizing the meaning of $z$ at every occurrence. Let $L^2=L^2(\mathbb{T}^d, \sigma)$ denote the standard Lebesgue space over $\mathbb{T}^d$, where $\sigma$ stands for the normalized Haar measure on the torus $\mathbb{T}^d$. It has been established that the set $\{e_{m_1, \dots, m_d}(z_1, \dots, z_d)=\prod_{i=1}^{n}z_i^{m_i}: m_1, \dots, m_d\in \mathbb{Z}\}$ forms an orthonormal basis for $L^2$. If a function $f$ belongs to $L^2$, it can be expressed by using its Fourier coefficients $\hat{f}(m_1, \dots, m_d)$ as 
$$f(z_1, \dots, z_d)=\sum_{m_1, \dots, m_d=-\infty}^{\infty}\widehat{f}(m_1, \dots, m_d)\prod_{i=1}^{d}z_i^{m_i}.$$
Furthermore, $L^\infty=L^\infty(\mathbb{T}^d)$ as the Banach space containing all essentially bounded functions defined on $\mathbb{T}^d$. 
\par
The Hardy space $\mathcal{H}^2(\mathbb{D}^d)$
 on the polydisk $\mathbb{D}^d$ refers to the space of analytic functions $f$
 defined on the polydisk $\mathbb{D}^d$ (a generalization of the disk to higher dimensions) that satisfies a specific condition related to their norm.

$$\|f\|^2=\sup_{0<r<1}\int_{\mathbb{T}^d}|f(r\xi)|^2d\sigma(\xi)<\infty.$$

It is well-known that for a function $f$ belonging to Hardy space $\mathcal{H}^2(\mathbb{D}^d)$, there exists a boundary function $\lim\limits_{r\rightarrow 1}f(r\xi)$, which is an element of $L^2$, for almost every $\xi\mathbb{T}^d$. The poisson extension of this boundary function coincides with $f$ on the polydisk $\mathbb{D}^d$. By identifying a function in $\mathcal{H}^2(\mathbb{D}^d)$ with its boundary function, the space $\mathcal{H}^2(\mathbb{D}^d)$ is a subset of $L^2$. So, the norm of $f\in \mathcal{H}^2$ can be denoted as $$\|f\|^2=\int_{\mathbb{T}^d}|f(\xi)|^2d\sigma(\xi).$$ 
The reproducing kernel function for Hardy space $\mathcal{H}^2(\mathbb{D}^d)$ is defined as 
$$K_w(z)=K_{w_1, \dots, w_d}(z_1, \dots, z_d)=\prod_{j=1}^{d}K_{w_j}(z_j)=\prod_{j=1}^{d}\frac{1}{(1-\overline{w_j}z_j)},$$
where $z=(z_1, \dots, z_d), ~w=(w_1, \dots, w_d)\in \mathbb{D}^d$ and $K_{w_j}(z_j)$'s are reproducing kernel functions for $\mathcal{H}^2(\mathbb{D})$. The normalized reproducing kernel for $\mathcal{H}^2(\mathbb{D}^d)$ is denoted by $\widehat{k}_w(z)=\prod_{j=1}^{d}e_{w_j}(z_j)$, with $e_{w_1}(z_1), \dots, e_{w_d}(z_d)$ being the normalized reproducing kernel functions for $\mathcal{H}^2(\mathbb{D})$. In a similar way, the reproducing kernel on the Hilbert space $\mathcal{H}_\gamma (\gamma\in \mathbb{N})$ over the polydisk is 
   
 $$K_w(z)=K_{w_1, \dots, w_d}(z_1, \dots, z_d)=\prod_{j=1}^{d}K_{w_j}(z_j)=\prod_{j=1}^{d}\frac{1}{(1-\overline{w_j}z_j)^\gamma},$$
 where $z=(z_1, \dots, z_d), ~w=(w_1, \dots, w_d)\in \mathbb{D}^d$ and $K_{w_j}(z_j)$'s are reproducing kernel functions for $\mathcal{H}_\gamma(\mathbb{D})$. The normalized reproducing kernel for $\mathcal{H}_\gamma(\mathbb{D}^d)$ is denoted by $\widehat{k}_w(z)=\prod_{j=1}^{d}e_{w_j}(z_j)$, with $e_{w_1}(z_1),  \dots, e_{w_d}(z_d)$ being the normalized reproducing kernel functions for $\mathcal{H}_\gamma(\mathbb{D})$.
 So,
  $$\widehat{k}_w(z)=\prod_{j=1}^{d}\frac{(1-|w_j|^2)^{\frac{\gamma}{2}}}{(1-\overline{w_j}z_j)^\gamma}.$$

 Let $\psi: \mathbb{D}^d\longrightarrow \mathbb{C}$ be an  analytic function and $\varphi: \mathbb{D}^d\longrightarrow \mathbb{D}^d$ be an analytic self-map of the polydisk $\mathbb{D}^d$ given by $\psi(z_1, \dots, z_d)=(\psi_1(z_1), \dots, \psi_d(z_d))$, where $\psi_1, \dots, \psi_d$ are analytic self-maps on the unit disk $\mathbb{D}$. The weighted composition operator $W_{\psi, \varphi}$ on $\mathcal{H}_\gamma(\mathbb{D}^d)$ defined as follows:
 $$ W_{\psi, \varphi} f(z)=\psi(z)f(\varphi(z)), ~f\in \mathcal{H}_\gamma(\mathbb{D}^d),$$
 that is, 
 $$W_{\psi, \varphi} f(z_1,\dots, z_d)=\psi(z_1,\dots, z_d)f(\varphi_1(z_1), \dots, \varphi_d(z_d)),~\mbox{for}~f\in \mathcal{H}_\gamma(\mathbb{D}^d).$$

We set the generalized weighted composition operator as   $D_{\mfn,\psi,\varphi}f(z),$ is defined by

\begin{gather*}
D_{\mfn,\psi,\varphi}f(z)=\psi(z)\partial^{\mfn}f(\varphi(z))=\psi(z)\frac{\partial^{n_1+\cdots+n_d}}{\partial z_1^{n_1},\dots,\partial z_d^{n_d}}f(\varphi(z)),
\end{gather*}
where $f$ is an element of the Hilbert space $\mathcal{H}_\gamma(\mathbb{D}^d)$. If $\mfn =0$, we obtain the weighted composition operator $W_{\psi, \varphi}$.

 We denote $K_w^{[\mfn]}(z)$ as the reproducing kernel for evaluating the $\sum_{j=1}^{d}n_j$-th partial derivative at a specific point $w$, where
\begin{align*}
K_w^{[\mfn]}(z)&=\prod_{j=1}^d\frac{\gamma(\gamma+1)\dots(\gamma+n_j-1)z^\mfn}{(1-\overline{w_j}z_j)^{\gamma+n_j}}\\
& =\prod_{j=1}^d\frac{\gamma(\gamma+1)\dots(\gamma+n_j-1)z_j^{n_j}}{(1-\overline{w_j}z_j)^{\gamma+n_j}},
\end{align*}
where  $\mfn=(n_1, \dots, n_d)$, are non-negative integers, $z, w\in \mathbb{D}^d$, and $z^{\mfn}=z_1^{n_1},\dots, z_d^{n_d}.$

\subsection{Problems for investigation}
  Composition operators on various function spaces have been studied extensively during the past few decades (see \cite{Cowen_MacCluer_1995,Shapiro_1993} and the references therein). The traditional choice of topics studied with respect to composition operators includes: boundedness, compactness, essential norm, closed range, etc. With the recent advances in the study of complex symmetry, it makes sense to ask the following question: {\bf are there any complex symmetric weighted composition operators?}\\
\par 
For the last decade, the above question has been studied by Jung {\it et al.} \cite{Jung_Kim_Ko_Lee_JFAA_2014} as well as  Garcia and Hammond \cite{Garcia_Hammond}. Independently, they have obtainded a criterion for complex symmetric structure of $W_{\psi, \varphi}$ on Hardy spaces with respect to the what we know as the standard conjugation, {\it i.e.,} those of the form $$ \mathcal{J}f(z)=\overline{f(\overline{z})},$$
where $f$ belongs to the Hardy space. Also, a criterion for the complex symmetric structure of $W_{\psi, \varphi}$ was discovered on Fock space $\mathcal{F}^2(\mathbb{C})$ (see \cite{Hai_Khoi_JMAA_2016}) and $\mathcal{F}^2(\mathbb{C}^n)$ (see \cite{Hai_Khoi_CVE_2018}) with respect to more general conjugation of the form 
$$\mathcal{A}_{u, v}f(z)=u(z)\overline{f(\overline{v(z)})},$$
where $u$ and $v$ are entire functions and $f$ belongs to the Fock space. Inspired by these recent results, Lim and Khoi \cite{Lim_Khoi_JMAA_2018} have attempted to study the complex symmetric structure of $W_{\psi, \varphi}$ on $\mathcal{H}_\gamma(\mathbb{D})$ with respect to conjuations of the form 
$$\mathcal{A}_{u, v}f(z)=u(z)\overline{f(\overline{v(z)})},$$
where $u: \mathbb{D}\longrightarrow \mathbb{C}$ and $v: \mathbb{D}\longrightarrow \mathbb{D}$ are holomorphic functions and $f\in \mathcal{H}_\gamma(\mathbb{D})$.\\

Now, the weighted composition conjuation on $\mathcal{H}_\gamma(\mathbb{D}^d)$ is of the form 
$$\mathcal{A}_{u, v}f(z)=u(z)\overline{f(\overline{v(z)})},~f\in \mathcal{H}_\gamma(\mathbb{D}^d),$$
that is, 
$$\mathcal{A}_{u, v} f(z_1,\dots, z_d)=u(z_1,\dots, z_d)\overline{f(\overline{v_1(z_1)}, \dots, \overline{v_d(z_d)})},~\mbox{for}~f\in \mathcal{H}_\gamma(\mathbb{D}^d).$$
If the function $u$ is identically equal to $1$, and $v$ is the identity map of the polydisk $\mathbb{D}^d$, the weighted composition conjugaton $\mathcal{A}_{u, v}$ simplifies to the standard conjugation $\mathcal{J}$, defined as $\mathcal{J}f(z)=\overline{f(\overline{z})}$. One can easily see that $\mathcal{A}_{u, v}=\mathcal{W}_{u, v}\mathcal{J}$. For a more comprehensive understanding of weighted composition conjugations, we refer to \cite{Lim_Khoi_JMAA_2018} and references therein. Glancing through the literature, the following natural question aries in our mind:\\ \\
{\bf Question-$1$: What is the complex symmetric structure of  weighted composition differentiation operators $D_{\mfn,\psi,\varphi}$ with respect to conjugations on $\mathcal{H}_\gamma(\mathbb{D}^d)$?}  \\

The present article narrows down to the precise question of classifying
weighted composition-differentiation operators $D_{\mfn,\psi,\varphi}$ acting on $\mathcal{H}_\gamma(\mathbb{D}^d)$, which are complex symmetric with respect to conjugations. The main focus of this paper is to provide all possible answers of the queries Question-1 to Question-4
raised above and below.\\ \\
The organization of the paper is as follows. In section 2, we obtain explicit conditions for weighted composition differentiation operators $D_{\mfn,\psi,\varphi}$ to be complex symmetry with respect to standard conjugation $\mathcal{J}$ on $\mathcal{H}_\gamma(\mathbb{D}^d)$. Furthermore, we establish both necessary and sufficient conditions for $D_{\mfn,\psi,\varphi}$ to be self-adjoint. We also characterize generalized weighted composition differentiation  operators $M_{n, \psi, \varphi}=\sum_{j=1}^{n}a_jD_{j,\psi_j, \varphi},$
(where $a_j\in \mathbb{C}$ for $j=1, 2, \dots, n$)  which are complex symmetry with respect to a weighted composition conjugation $C_{\mu, \xi}$ (Theorem \ref{Thm_3}) on  $\mathcal{H}_\gamma(\mathbb{D})$. We also explore the self-adjoint property for generalized weighted composition differentiation  operators $M_{n, \psi, \varphi}$ (Theorem \ref{Thm_4}) on $\mathcal{H}_\gamma(\mathbb{D})$. 
Section 3 is devoted to applications, and is of two folds. The first part presents linear composition operators not the  identities, the origin lies in the closure of the numerical range while the next part focuses on characterizing convexity of the Berezin range for composition operators acting on $\mathcal{H}_\gamma(\mathbb{D})$. Further, we geometrically illustrate the region of convexity for several sets of parameters using Python.

\section{Results and discussions}

\subsection{Generalized complex symmetric composition differentiation operators on $\mathcal{H}_\gamma(\mathbb{D}^{d})$}

In this subsection, we characterize complex symmetric composition differentiation operators on $\mathcal{H}_\gamma(\mathbb{D}^{d})$ over the unit polydisc with respect to conjugation $\mathcal{J}$.
\begin{lemma}
    Let $\mfn$ be a non-negative integer and $w\in \mathbb{D}^d$. Let $\psi: \mathbb{D}^d\longrightarrow \mathbb{C}$ be a non-zero analytic function and $\varphi: \mathbb{D}^d\longrightarrow \mathbb{D}^d$ with $\varphi(z_1, z_2, \dots, z_d)=(\varphi_1(z_1), \varphi_2(z_2), \dots, \varphi_d(z_d))$ be an analytic self map of $\mathbb{D}^d$ such that $D_{\mfn,\psi,\varphi}$ is bounded on $\mathcal{H}_\gamma(\mathbb{D}^d)$.
    Then $$  D_{\mfn,\psi, \varphi}^*K_w(z)=\overline{\psi(w)}K_{\varphi(w)}^{[\mfn]}(z).$$
\end{lemma}
\begin{proof}
    Let $f\in \mathcal{H}_\gamma(\mathbb{D}^d)$, then
    \begin{align*}
        \langle f,  D_{\mfn,\psi, \varphi}^*K_w \rangle_{\mathcal{H}_\gamma(\mathbb{D}^d)}&=  \langle  D_{\mfn,\psi, \varphi}f, K_w \rangle_{\mathcal{H}_\gamma(\mathbb{D}^d)}\\
        &=\langle \psi \partial^n f(\varphi) , K_w \rangle_{\mathcal{H}_\gamma(\mathbb{D}^d)}\\
        &=\psi(w) \partial^n f(\varphi(w)) \\
        &=\langle  f, \overline{\psi(w)}K_{\varphi(w)}^{[\mfn]} \rangle_{\mathcal{H}_\gamma(\mathbb{D}^d)}.
    \end{align*}
    So, we have
    $$\langle f,  D_{\mfn,\psi, \varphi}^*K_w \rangle_{\mathcal{H}_\gamma(\mathbb{D}^d)}=\langle  f, \overline{\psi(w)}K_{\varphi(w)}^{[\mfn]} \rangle_{\mathcal{H}_\gamma(\mathbb{D}^d)}~~\mbox{for~ all} f\in \mathcal{H}_\gamma(\mathbb{D}^d)    .$$
    Then
     $$  D_{\mfn,\psi, \varphi}^*K_w(z)=\overline{\psi(w)}K_{\varphi(w)}^{[\mfn]}(z).$$
 \end{proof}

The following result investigates the complex symmetric structure of weighted composition differentiation operators $D_{\mfn,\psi,\varphi}$ with respect to the conjugation $\mathcal{J}$ on $\mathcal{H}_\gamma(\mathbb{D}^d)$.
\begin{theorem}\label{Thm_1}
Let $\gamma\in \mathbb{N}^+$ and $\mfn$ be a non-negative integer. Let $\psi: \mathbb{D}^d\longrightarrow \mathbb{C}$ be a non-zero analytic function and $\varphi: \mathbb{D}^d\longrightarrow \mathbb{D}^d$ with $\varphi(z_1,  \dots, z_d)=(\varphi_1(z_1), \dots, \varphi_d(z_d))$ be an analytic self map of $\mathbb{D}^d$ such that $D_{\mfn,\psi,\varphi}$ is bounded on $\mathcal{H}_\gamma(\mathbb{D}^d)$.
    Then $D_{\mfn,\psi,\varphi}$ is complex symmetric with respect to the conjugation $\mathcal{J}$ if, and only if,
\begin{gather*}
    \psi(z_1, \dots, z_d)=a\prod_{j=1}^dz_j^{n_j}\big(1-z_j\varphi_j(0)\big)^{-\gamma-n_j},
\end{gather*}
and 
\begin{gather*}
        \varphi_t(z_t)=\varphi_t(0)+\big(1-z_t\varphi_t(0)\big)^{-1}z_t\varphi_t^{\prime}(0),~ \mbox{for}~ t=1, \dots, d,
\end{gather*}
where $a=\partial^n\psi(0, \dots, 0)$.
\end{theorem}
\begin{proof}
        Suppose that $D_{\mfn,\psi,\varphi}$ is complex symmetric with respect to the conjugation $\mathcal{J}$, then we have
    \begin{align}\label{Them_Eqn_1}
         D_{\mfn,\psi, \varphi}\mathcal{J}K_w(z)= \mathcal{J}D_{\mfn,\psi, \varphi}^*K_w(z)
    \end{align}
    for all $z=(z_1, \dots, z_d)$, $w=(w_1, \dots, w_d)\in \mathbb{D}^d$.
    The left hand side of \eqref{Them_Eqn_1} gives us 
    \begin{align}\label{Them_Eqn_2}
         D_{\mfn,\psi, \varphi}\mathcal{J}K_w(z)&=\psi(z)\partial^n\mathcal{J}K_w(\varphi(z))\nonumber\\
         &=\psi(z_1, \dots, z_d)\partial^n\mathcal{J}K_{(w_1, \dots, w_d)}(\varphi_1(z_1), \dots, \varphi_d(z_d))\nonumber\\
         &=\psi(z_1, \dots, z_d)\partial^n\overline{K_{(w_1, \dots, w_d)}(\overline{\varphi_1(z_1)}, \dots, \overline{\varphi_d(z_d))}}\nonumber\\
          &=\psi(z_1, \dots, z_d)\partial^n\overline{K_{w_1}(\overline{\varphi_1(z_1)})},\dots, \overline{K_{w_d}(\overline{\varphi_d(z_d)})}\nonumber\\
           &=\psi(z_1, \dots, z_d)\partial^n\prod_{j=1}^d(1-w_j\varphi_j(z_j))^{-\gamma}\nonumber\\
            &=\psi(z_1, \dots, z_d)\prod_{j=1}^d(\gamma+n_j-1)!w_j^{n_j}(1-w_j\varphi_j(z_j))^{-\gamma-n_j}.
    \end{align}
    On the other hand, we have 
    \begin{align}\label{Them_Eqn_3}
    	\mathcal{J} D_{\mfn,\psi, \varphi}^*K_w(z)&=\mathcal{J}\overline{\psi(w)}K_{\varphi(w)}^{[\mfn]}(z)\nonumber\\
    	&=\mathcal{J}\overline{\psi(w)}\prod_{j=1}^d(\gamma+n_j-1)!z_j^{n_j}(1-z_j\overline{\varphi_j(w_j)})^{-\gamma-n_j}\nonumber\\
    	&=\psi(w_1, \dots, w_d)\prod_{j=1}^d(\gamma+n_j-1)!z_j^{n_j}(1-z_j\varphi_j(w_j))^{-\gamma-n_j}.
    \end{align}
    Hence from \eqref{Them_Eqn_1}, we have 
    \begin{align}\label{Them_Eqn_4}
    \psi(z_1, \dots, z_d)&\prod_{j=1}^d(\gamma+n_j-1)!w_j^{n_j}(1-w_j\varphi_j(z_j))^{-\gamma-n_j}\nonumber\\
    &=\psi(w_1, \dots, w_d)\prod_{j=1}^d(\gamma+n_j-1)!z_j^{n_j}(1-z_j\varphi_j(w_j))^{-\gamma-n_j}.
    \end{align}
    Letting $z=0, i.e., z_1=0,\dots,z_d=0$ in \eqref{Them_Eqn_4}, we get $\psi(0)=0$, that means $\psi(0, z_2, \dots,z_d)=0, \psi(z_1, 0 \dots,z_d)=0,\dots,\psi(z_1, z_2 \dots,z_{d-1}, 0)=0 $,  $\psi(0, \dots, 0)=0$ and so we can write $\psi(z)=g(z_1, \dots, z_d)\prod_{j=1}^dz_j^{b_j}$ and $g$ is analytic with $g(0, z_2, \dots, z_d)\neq 0, g(z_1, 0, \dots, z_d)\neq 0,\dots, g(z_1, \dots, z_{d-1},0)\neq 0$, $g(0,  \dots, 0)\neq 0$ and $b_j\in \mathbb{N}, j\in \{ 
    1, \dots, d\}$.
    Then, we get from \eqref{Them_Eqn_4} that
    \begin{align}\label{Them_Eqn_5}
   g(z_1, \dots, z_d)&\prod_{j=1}^d(\gamma+n_j-1)!w_j^{n_j}z_j^{b_j}(1-w_j\varphi_j(z_j))^{-\gamma-n_j}\nonumber\\
    &=g(w_1, \dots, w_d)\prod_{j=1}^d(\gamma+n_j-1)!w_j^{b_j}z_j^{n_j}(1-z_j\varphi_j(w_j))^{-\gamma-n_j}.
    \end{align}
    We show that $n_j=b_j$ for $j\in \{ 
    1, \dots, d\}$.\\ \\
    \underline{\bf Case-I:} If $n_j>b_j$, then it follows from \eqref{Them_Eqn_5} that 
     \begin{align}\label{Them_Eqn_6}
    g(z_1, \dots, z_d)&\prod_{j=1}^dw_j^{n_j-b_j}(1-w_j\varphi_j(z_j))^{-\gamma-n_j}\nonumber\\
   & =g(w_1, \dots, w_d)\prod_{j=1}^dz_j^{n_j-b_j}(1-z_j\varphi_j(w_j))^{-\gamma-n_j},
    \end{align}
    which is impossible because by letting $z_1=0,\dots,z_d=0$ in \eqref{Them_Eqn_6}, we get    $g(0, z_2, \dots, z_d)= 0, g(z_1, 0, \dots, z_d)= 0,\dots, g(z_1, \dots, z_{d-1},0)= 0$ and $g(0,  \dots, 0)= 0$ which is a contradiction. This contradiction is due to the fact that  $g(0, z_2, \dots, z_d)\neq 0, g(z_1, 0, \dots, z_d)\neq 0,\dots, g(z_1, \dots, z_{d-1},0)\neq 0$ and $g(0,  \dots, 0)\neq 0$. \\ \\
    \underline{\bf Case-II:} If $n_j<b_j$, then \eqref{Them_Eqn_5} becomes the following
     \begin{align}\label{Them_Eqn_7}
    g(z_1, \dots, z_d)&\prod_{j=1}^dz_j^{b_j-n_j}(1-w_j\varphi_j(z_j))^{-\gamma-n_j}\nonumber\\
    &=g(w_1, \dots, w_d)\prod_{j=1}^dw_j^{b_j-n_j}(1-z_j\varphi_j(w_j))^{-\gamma-n_j},
    \end{align}
    which is not possible because by letting $w_1=0,\dots,w_d=0$ in \eqref{Them_Eqn_7}, we get    $g(z,  \dots, z_d)\equiv 0$ on $\mathbb{D}^d$ which is a contradiction. Thus, we must have $n_j=b_j$ for all  $j\in \{ 
    1, \dots, d\}$ and hence \eqref{Them_Eqn_5} reduces to the following   
     \begin{align}\label{Them_Eqn_8}
    g(z_1, \dots, z_d)\prod_{j=1}^d(1-w_j\varphi_j(z_j))^{-\gamma-n_j}
    =g(w_1, \dots, w_d)\prod_{j=1}^d(1-z_j\varphi_j(w_j))^{-\gamma-n_j}.
    \end{align}
    Letting $w_j=0$ for all $j\in \{ 
    1, \dots, d\}$ in \eqref{Them_Eqn_8}, we obtain
    \begin{gather}\label{Them_Eqn_9}
   g(z_1, \dots, z_d)=g(0,  \dots, 0)\prod_{j=1}^d\big(1-z_j\varphi_j(0)\big)^{-\gamma-n_j}.
    \end{gather}
    It follows that 
    \begin{align}\label{Them_Eqn_10}
    \psi(z_1, \dots, z_d)&=g(z_1, \dots, z_d)\prod_{j=1}^dz_j^{n_j}\nonumber\\
    &=g(0,  \dots, 0)\prod_{j=1}^dz_j^{n_j}\big(1-z_j\varphi_j(0)\big)^{-\gamma-n_j}.
    \end{align}
  Since $g(0,  \dots, 0)\neq 0$, it follows from \eqref{Them_Eqn_4} and \eqref{Them_Eqn_10} that 
  \begin{align}\label{Them_Eqn_11}
  	\prod_{j=1}^d\big(1-w_j\varphi_j(0)\big)^{\gamma+n_j}\times	&\prod_{j=1}^d\big(1-z_j\varphi_j(w_j)\big)^{\gamma+n_j}\nonumber\\
  	&=\prod_{j=1}^d\big(1-z_j\varphi_j(0)\big)^{\gamma+n_j}\times	\prod_{j=1}^d\big(1-w_j\varphi_j(z_j)\big)^{\gamma+n_j}.
  	  \end{align}
  	  Through differentiating both sides with respect to the variable $z_t$, $(t\in\{1, \dots, d\})$ in \eqref{Them_Eqn_11} we have
  	  \begin{align}\label{Them_Eqn_12}
  	& (\gamma+n_t) \prod_{j=1}^d\big(1-w_j\varphi_j(0)\big)^{\gamma+n_j}\times	\prod_{j\neq t}^d\big(1-z_j\varphi_j(w_j)\big)^{\gamma+n_j}\big(1-z_t\varphi_t(w_t)\big)^{\gamma+n_t-1}(-\varphi_t(w_t))\nonumber\\
  	 &= (\gamma+n_t)\prod_{j\neq t}^d\big(1-z_j\varphi_j(0)\big)^{\gamma+n_j}\big(1-z_t\varphi_t(0)\big)^{\gamma+n_t-1}(-\varphi_t(0))\times	\prod_{j=1}^d\big(1-w_j\varphi_j(z_j)\big)^{\gamma+n_j}\nonumber\\
  	 &\hspace{2cm}+(\gamma+n_t)\prod_{j=1}^d\big(1-z_j\varphi_j(0)\big)^{\gamma+n_j}\prod_{j\neq t}^d\big(1-w_j\varphi_j(z_j)\big)^{\gamma+n_j}\nonumber\\
     &\hspace{7cm}\times	\big(1-w_t\varphi_t(z_t)\big)^{\gamma+n_t-1}\times(-w_t\varphi_t^{\prime}(z_t)).
  	  \end{align}
  	  Evaluating at $z_j=0$ for all $j=1, \dots, d$ and $w_j=0$ for $j\neq t$ in \eqref{Them_Eqn_12}, we obtain 
  	  \begin{align}\label{Them_Eqn_13}
  	  	\big(1-w_t\varphi_t(0)\big)^{\gamma+n_t}&(\varphi_t(w_t))\nonumber\\
  	  	&=	\big(1-w_t\varphi_t(0)\big)^{\gamma+n_t}(\varphi_t(0))+\big(1-w_t\varphi_t(0)\big)^{\gamma+n_t-1}(w_t\varphi_t^{\prime}(0)),
  	  \end{align} 
  	  which implies that 
  	  \begin{gather*}
  	  \varphi_t(w_t)=\varphi_t(0)+\big(1-w_t\varphi_t(0)\big)^{-1}w_t\varphi_t^{\prime}(0),~ \mbox{for}~ t=1, \dots, d.
  	  \end{gather*}
  	  Thus,
  	  \begin{gather*}
  	  \psi(z_1, \dots, z_d)=a\prod_{j=1}^dz_j^{n_j}\big(1-z_j\varphi_j(0)\big)^{-\gamma-n_j},
  	  \end{gather*}
  	  and 
  	  \begin{gather*}
  	  \varphi_t(z_t)=\varphi_t(0)+\big(1-z_t\varphi_t(0)\big)^{-1}z_t\varphi_t^{\prime}(0),~ \mbox{for}~ t=1, \dots, d,
  	  \end{gather*}
  	  where $$a=\partial^n\psi(0, \dots, 0)=\frac{\partial^{n_1+\cdots+n_d}}{\partial z_1^{n_1},\dots, \partial z_d^{n_d}}\psi(0, \dots, 0).$$
  	  
  	  Conversly assume that 
  	  \begin{gather*}
  	  \psi(z_1, \dots, z_d)=a\prod_{j=1}^dz_j^{n_j}\big(1-z_j\varphi_j(0)\big)^{-\gamma-n_j},
  	  \end{gather*}
  	  and 
  	  \begin{gather*}
  	  \varphi_t(z_t)=\varphi_t(0)+\big(1-z_t\varphi_t(0)\big)^{-1}z_t\varphi_t^{\prime}(0),~ \mbox{for}~ t=1, \dots, d,
  	  \end{gather*}
  	  where $a=\partial^n\psi(0, \dots, 0)$.
  	  From \eqref{Them_Eqn_2} and \eqref{Them_Eqn_3}, we have  
  	  
  	   \begin{align}\label{Them_Eqn_14}
  	   D_{\mfn,\psi, \varphi}&\mathcal{J}K_w(z)
  	 = \psi(z_1, \dots, z_d)\prod_{j=1}^d(\gamma+n_j-1)!w_j^{n_j}(1-w_j\varphi_j(z_j))^{-\gamma-n_j}\nonumber\\
  	  =&a\prod_{j=1}^dz_j^{n_j}\big(1-z_j\varphi_j(0)\big)^{-\gamma-n_j}\times\prod_{j=1}^d(\gamma+n_j-1)!w_j^{n_j}\nonumber\\
  	 ~~~~~~~~~~~~~~~~~~~~&\left(1-w_j\left(\varphi_j(0)+\big(1-z_j\varphi_j(0)\big)^{-1}z_j\varphi_j^{\prime}(0)\right)\right)^{-\gamma-n_j}\nonumber\\
  	  =&a\prod_{j=1}^d(\gamma+n_j-1)!z_j^{n_j}w_j^{n_j}\left(1-z_j\varphi_j(0)-w_j\left((1-z_j\varphi_j(0))\varphi_j(0)+z_j\varphi_j^{\prime}(0)\right) \right)^{-\gamma-n_j}.
  	  \end{align} 
  	  and 
  	   \begin{align}\label{Them_Eqn_15}
  	  &  \mathcal{J}D_{\mfn,\psi, \varphi}^*K_w(z)
  	  = \psi(w_1, \dots, w_d)\prod_{j=1}^d(\gamma+n_j-1)!z_j^{n_j}(1-z_j\varphi_j(w_j))^{-\gamma-n_j}\nonumber\\
  	  =&a\prod_{j=1}^dw_j^{n_j}\big(1-w_j\varphi_j(0)\big)^{-\gamma-n_j}\times\prod_{j=1}^d(\gamma+n_j-1)!z_j^{n_j}\nonumber\\
  	  &\left(1-z_j\left(\varphi_j(0)+\big(1-w_j\varphi_j(0)\big)^{-1}w_j\varphi_j^{\prime}(0)\right)\right)^{-\gamma-n_j}\nonumber\\
  	  =&a\prod_{j=1}^d(\gamma+n_j-1)!z_j^{n_j}w_j^{n_j}\left(1-w_j\varphi_j(0)-z_j\left((1-w_j\varphi_j(0))\varphi_j(0)+w_j\varphi_j^{\prime}(0)\right) \right)^{-\gamma-n_j},
  	  \end{align} 
  	   for all $z=(z_1, \dots, z_d)$, $w=(w_1, \dots, w_d)\in \mathbb{D}^d$.
  	   Since 
  	   \begin{align}\label{Them_Eqn_16}
  	   &1-z_j\varphi_j(0)-w_j\left((1-z_j\varphi_j(0))\varphi_j(0)+z_j\varphi_j^{\prime}(0)\right)\nonumber\\
  	   &=1-z_j\varphi_j(0)-w_j\varphi_j(0)-w_jz_j\varphi_j^2(0)+w_jz_j\varphi_j^{\prime}(0)\nonumber\\
  	   &=1-w_j\varphi_j(0)-z_j\left((1-w_j\varphi_j(0))\varphi_j(0)+w_j\varphi_j^{\prime}(0)\right),
  	   \end{align}
  	   it follows from \eqref{Them_Eqn_14} and \eqref{Them_Eqn_15} that 
  	   $D_{\mfn,\psi, \varphi}\mathcal{J}K_w(z)=\mathcal{J}D_{\mfn,\psi, \varphi}^*K_w(z)$. Hence $D_{\mfn,\psi,\varphi}$ is complex symmetric with respect to the conjugation $\mathcal{J}$. This completes the proof.
\end{proof}
\begin{remark}
If we set $\gamma=1~(\mbox{or}~ 2)$, we obtain explicit forms of $\psi$ and $\varphi$ for which $D_{\mfn,\psi,\varphi}$ is complex symmetric with respect to the usual conjugation on the Hardy and Bergman space on the polydisk. More precisely, we have the following corollary, which establishes complex symmetry of $D_{\mfn,\psi,\varphi}$ on the Bergman space on the polydisk.
  \end{remark}
\begin{corollary}
Let  $\mfn$ be a non-negative integer. Let $\psi: \mathbb{D}^d\longrightarrow \mathbb{C}$ be a non-zero analytic function and $\varphi: \mathbb{D}^d\longrightarrow \mathbb{D}^d$ with $\varphi(z_1,  \dots, z_d)=(\varphi_1(z_1), \dots, \varphi_d(z_d))$ be an analytic self map of $\mathbb{D}^d$ such that $D_{\mfn,\psi,\varphi}$ is bounded on $\mathcal{L}_a^2(\mathbb{D}^d)$.
Then $D_{\mfn,\psi,\varphi}$ is complex symmetric with respect to the conjugation $\mathcal{J}$ if, and only if,
\begin{gather*}
\psi(z_1, \dots, z_d)=a\prod_{j=1}^dz_j^{n_j}\big(1-z_j\varphi_j(0)\big)^{-n_j-2},
\end{gather*}
and 
\begin{gather*}
\varphi_j(z_j)=\varphi_j(0)+\big(1-z_j\varphi_j(0)\big)^{-1}z_j\varphi_j^{\prime}(0),~ \mbox{for}~ j=1, \dots, d,
\end{gather*}
where $a=\partial^n\psi(0, \dots, 0)$.
\end{corollary}
\begin{remark}
	\begin{enumerate}
	\item[(i)] 	As a special case of the Theorem \ref{Thm_1}, we obtain the complex symmetric structure of the composition differentiation operator on the Hardy space on the polydisk, which has been recently proved by Ahamed {\it et al}. \cite[Theorem 2.2]{Ahamed_Allu_Rahman_2023}.
		\item[(ii)] Complex symmetric structure of the weighted composition differentiation operator $D_{\mfn,\psi,\varphi}$ on $\mathcal{H}_\gamma(\mathbb{D})$ on unit disk case, (see \cite{Lo_Loh_JMAA_2023} and for weighted composition operators $W_{\psi,\varphi}$ on $\mathcal{H}_\gamma(\mathbb{D})$ (see \cite [Proposition 2.3]{Lim_Khoi_JMAA_2018}). Similarly, dealing with complex symmetry of $D_{\psi,\varphi}$ on Hardy and Bergman space on the unit disk (see \cite{Han_Wang_JMAA_2021,Han_Wang_CAOT_2021}).
	\end{enumerate}
   
	 \end{remark}

 In Theorem \ref{Thm_2}, we obtain a condition that is both necessary and sufficient for the bounded weighted composition differentiation operator $D_{\mfn,\psi,\varphi}$ to satisfy self-adjoint properties on  $\mathcal{H}_\gamma(\mathbb{D}^d)$.
 \begin{theorem}\label{Thm_2}
 	Let $\gamma\in \mathbb{N}^+$ and $\mfn$ be a non-negative integer. Let $\psi: \mathbb{D}^d\longrightarrow \mathbb{C}$ be a non-zero analytic function and $\varphi: \mathbb{D}^d\longrightarrow \mathbb{D}^d$ with $\varphi(z_1,  \dots, z_d)=(\varphi_1(z_1), \dots, \varphi_d(z_d))$ be an analytic self map of $\mathbb{D}^d$ such that $D_{\mfn,\psi,\varphi}$ is bounded on $\mathcal{H}_\gamma(\mathbb{D}^d)$.
 	Then $D_{\mfn,\psi,\varphi}$ is self adjoint if, and only if,
 	\begin{gather*}
 	\psi(z_1, \dots, z_d)=a\prod_{j=1}^dz_j^{n_j}\Big(1-z_j\overline{\varphi_j(0)}\Big)^{-\gamma-n_j},
 	\end{gather*}
 	and 
 	\begin{gather*}
 	\varphi_t(z_t)=\varphi_t(0)+\Big(1-z_t\overline{\varphi_t(0)}\Big)^{-1}z_t\overline{\varphi_t^{\prime}(0)},~ \mbox{for}~ t=1, \dots, d,
 	\end{gather*}
 	where $a=\overline{g(0, \dots, 0)}$.
 \end{theorem}
 \begin{proof}
  	Suppose that $D_{\mfn,\psi,\varphi}$ is self adjoint. Then we have
 	\begin{align}\label{Them_2_Eqn_1}
 	D_{\mfn,\psi, \varphi}^*K_w(z)= D_{\mfn,\psi, \varphi}K_w(z)
 	\end{align}
 	for all $z=(z_1, \dots, z_d)$, $w=(w_1, \dots, w_d)\in \mathbb{D}^d$.
 	The left hand side of \eqref{Them_2_Eqn_1} gives us 
 	\begin{align}\label{Them_2_Eqn_2}
 	D_{\mfn,\psi, \varphi}^*K_w(z)&=\overline{\psi(w)}K_{\varphi(w)}^{[n]}(z)\nonumber\\
 	&=\overline{\psi(w_1, \dots, w_d)}\prod_{j=1}^d(\gamma+n_j-1)!z_j^{n_j}\Big(1-z_j\overline{\varphi_j(w_j)}\Big)^{-\gamma-n_j}.
 	\end{align}
 	On the other hand, we have 
 	\begin{align}\label{Them_2_Eqn_3}
 	 D_{\mfn,\psi, \varphi}K_w(z)=\psi(z_1, \dots, z_d)\prod_{j=1}^d(\gamma+n_j-1)!\overline{w_j}^{n_j}\Big(1-\overline{w_j}\varphi_j(z_j)\Big)^{-\gamma-n_j}.
 	\end{align}
 Therefore, from \eqref{Them_2_Eqn_2} and \eqref{Them_2_Eqn_3}, we have 
 	\begin{align}\label{Them_2_Eqn_4}
 	\psi(z_1, \dots, z_d)&\prod_{j=1}^d(\gamma+n_j-1)!\overline{w_j}^{n_j}(1-\overline{w_j}\varphi_j(z_j))^{-\gamma-n_j}\nonumber\\
 	&=\overline{\psi(w_1, \dots, w_d)}\prod_{j=1}^d(\gamma+n_j-1)!z_j^{n_j}(1-z_j\overline{\varphi_j(w_j)})^{-\gamma-n_j}.
 	\end{align}
 
 By applying similar method as used in Theorem \ref{Thm_1}, we assume that   
  $\psi(z_1, \dots, z_d)=g(z_1, \dots, z_d)\prod_{j=1}^dz_j^{b_j}$ and $g$ is analytic with $g(0, z_2, \dots, z_d)\neq 0$, $g(z_1, 0, \dots, z_d)\neq 0$, $\dots, g(z_1, \dots, z_{d-1},0)\neq 0$, $g(0,  \dots, 0)\neq 0$ and $b_j\in \mathbb{N}, j\in \{ 
 	1, \dots, d\}$.
 	Now, it follows from \eqref{Them_2_Eqn_4} and the method which we used in Theorem \ref{Thm_1} that
 	\begin{align}\label{Them_2_Eqn_5}
 	g(z_1, \dots, z_d)&\prod_{j=1}^d(\gamma+n_j-1)!(1-\overline{w_j}\varphi_j(z_j))^{-\gamma-n_j}\nonumber\\
 	&=g\overline{(w_1, \dots, w_d)}\prod_{j=1}^d(\gamma+n_j-1)!(1-z_j\overline{\varphi_j(w_j)})^{-\gamma-n_j}.
 	\end{align}
 	 	Letting $w_j=0$ for all $j=1, 2, \dots, d$ in \eqref{Them_2_Eqn_5}, we obtain
 	\begin{align*}
 	g(z_1, \dots, z_d)=g\overline{(0, \dots, 0)}\prod_{j=1}^d(1-z_j\overline{\varphi_j(0)})^{-\gamma-n_j}.
 	\end{align*}
 	It follows that  
 	\begin{align}\label{Eqn_2.22}
 		\psi(z_1, \dots, z_d)&=g(z_1, \dots, z_d)\prod_{j=1}^dz_j^{b_j}\nonumber\\
 		&=g\overline{(0, \dots, 0)}\prod_{j=1}^dz_j^{n_j}(1-z_j\overline{\varphi_j(0)})^{-\gamma-n_j}\nonumber\\
 			&=a\prod_{j=1}^dz_j^{n_j}(1-z_j\overline{\varphi_j(0)})^{-\gamma-n_j},
 	 	\end{align}
 	where $a=g\overline{(0, \dots, 0)}$. 	
 	Substituting \eqref{Eqn_2.22} in \eqref{Them_2_Eqn_4}, we obtain 	
 	 	\begin{align}\label{Them_2_Eqn_11}
 	a \prod_{j=1}^d \big(1-\overline{w_j}\varphi_j(0)\big)^{\gamma+n_j}\times	&\prod_{j=1}^d\big(1-z_j\overline{\varphi_j(w_j)}\big)^{\gamma+n_j}\nonumber\\
 	&=\overline{a}\prod_{j=1}^d\big(1-z_j\overline{\varphi_j(0)}\big)^{\gamma+n_j}\times	\prod_{j=1}^d\big(1-\overline{w_j}\varphi_j(z_j)\big)^{\gamma+n_j},
 	\end{align}
 for all $z=(z_1, \dots, z_d)$, $w=(w_1, \dots, w_d)\in \mathbb{D}^d$.
 
 If we put $w_1, \dots, w_d=0$, then we obtain $a=\overline{a}$, that is  $g\overline{(0, \dots, 0)}=g(0, \dots, 0) $, that is $g(0, \dots, 0)\in \mathbb{R}$.

 	By differentiating \eqref{Them_2_Eqn_11} both sides with respect to the variable $z_t$, $(t\in\{1, \dots, d\})$, we obtain 	
 	 	\begin{align}\label{Them_2_Eqn_12}
 	& (\gamma+n_t) \prod_{j=1}^d\big(1-\overline{w_j}\varphi_j(0)\big)^{\gamma+n_j}\times	\prod_{j\neq t}^d\big(1-z_j\overline{\varphi_j(w_j)}\big)^{\gamma+n_j}\big(1-z_t\overline{\varphi_t(w_t)}\big)^{\gamma+n_t-1}(-\overline{\varphi_t(w_t)})\nonumber\\
 	&= (\gamma+n_t)\prod_{j\neq t}^d\big(1-z_j\overline{\varphi_j(0)}\big)^{\gamma+n_j}\big(1-z_t\overline{\varphi_t(0)}\big)^{\gamma+n_t-1}(-\overline{\varphi_t(0)})\times	\prod_{j=1}^d\big(1-\overline{w_t}\varphi_j(z_j)\big)^{\gamma+n_j}\nonumber\\
 	&+(\gamma+n_t)\prod_{j=1}^d\big(1-z_j\overline{\varphi_j(0)}\big)^{\gamma+n_j}\prod_{j\neq t}^d\big(1-\overline{w_j}\varphi_j(z_j)\big)^{\gamma+n_j}\times	\big(1-\overline{w_t}\varphi_t(z_t)\big)^{\gamma+n_t-1}\times(-\overline{w_t}\varphi_t^{\prime}(z_t)),
 	\end{align}
 	 for all $z=(z_1, \dots, z_d)$, $w=(w_1, \dots, w_d)\in \mathbb{D}^d$.\\
 	 
 	Evaluating \eqref{Them_2_Eqn_12} at $z_j=0$ for all $j=1, \dots, d$ and $w_j=0$ for $j\neq t$, we get 
 	\begin{align}\label{Them_2_Eqn_13}
 	\big(1-\overline{w_t}\varphi_t(0)\big)^{\gamma+n_t}(\overline{\varphi_t(w_t)})=	\big(1-\overline{w_t}\varphi_t(0)\big)^{\gamma+n_t}(\overline{\varphi_t(0)})+\big(1-\overline{w_t}\varphi_t(0)\big)^{\gamma+n_t-1}(\overline{w_t}\varphi_t^{\prime}(0)),
 	\end{align} 
 	which implies that 
 	\begin{gather*}
 	\overline{\varphi_t(w_t)}=\overline{\varphi_t(0)}+\big(1-\overline{w_t}\varphi_t(0)\big)^{-1}\overline{w_t}\varphi_t^{\prime}(0),~ \mbox{for}~ t=1, \dots, d.
 	\end{gather*}
 	Thus,
 	 	\begin{align}\label{Them_2_Eqn_13_}
 	\varphi_t(w_t)=\varphi_t(0)+\big(1-w_t\overline{\varphi_t(0)}\big)^{-1}w_t\overline{\varphi_t^{\prime}(0)},~ \mbox{for}~ t=1, \dots, d.
 	\end{align}
 Differentiating \eqref{Them_2_Eqn_13_} partially with respect to $w_t$, we obtain 
 	\begin{align}\label{Them_2_Eqn_13.}
 \varphi_t^{\prime}(w_t)=\overline{\varphi_t^{\prime}(0)}\big(1-w_t\overline{\varphi_t(0)}\big)^{-2}.
 \end{align}
 Setting $w_t=0, (t=1, 2, \dots, d)$ in \eqref{Them_2_Eqn_13.} we obtain $\varphi_t^{\prime}(0)=\overline{\varphi_t^{\prime}(0)}$ which implies that $\varphi_t^{\prime}(0)\in \mathbb{R}$ for all $t=1, 2, \dots, d.$\\
 
 	Conversly let us assume that 
 	\begin{gather*}
 	\psi(z_1, \dots, z_d)=a\prod_{j=1}^dz_j^{n_j}\big(1-z_j\overline{\varphi_j(0)}\big)^{-\gamma-n_j},
 	\end{gather*}
 	and 
 	\begin{gather*}
 	\varphi_t(z_t)=\varphi_t(0)+\big(1-z_t\overline{\varphi_t(0)}\big)^{-1}z_t\overline{\varphi_t^{\prime}(0)},~ \mbox{for}~ t=1, \dots, d,
 	\end{gather*}
 	where $a=\overline{g(0, \dots, 0)}$.
 	From \eqref{Them_2_Eqn_2} and \eqref{Them_2_Eqn_3}, we have  
 	
 	\begin{align}\label{Them_2_Eqn_14}
 	  &D_{\mfn,\psi, \varphi}K_w(z)
 	= \psi(z_1, \dots, z_d)\prod_{j=1}^d(\gamma+n_j-1)!\overline{w_j}^{n_j}(1-\overline{w_j}\varphi_j(z_j))^{-\gamma-n_j}\nonumber\\
 	=&a\prod_{j=1}^dz_j^{n_j}\big(1-z_j\overline{\varphi_j(0)}\big)^{-\gamma-n_j}\times\prod_{j=1}^d(\gamma+n_j-1)!\overline{w_j}^{n_j}\nonumber\\
 	&\left(1-\overline{w_j}\left(\varphi_j(0)+\big(1-z_j\overline{\varphi_j(0)}\big)^{-1}z_j\overline{\varphi_j^{\prime}(0)}\right)\right)^{-\gamma-n_j}\nonumber\\
 	=&a\prod_{j=1}^d(\gamma+n_j-1)!z_j^{n_j}\overline{w_j}^{n_j}\left(1-z_j\overline{\varphi_j(0)}-\overline{w_j}\left((1-z_j\overline{\varphi_j(0)})\varphi_j(0)+z_j\overline{\varphi_j^{\prime}(0)}\right) \right)^{-\gamma-n_j}.
 	\end{align} 
 	and 
 	\begin{align}\label{Them_2_Eqn_15}
 	& D_{\mfn,\psi, \varphi}^*K_w(z)
 	= \overline{\psi(w_1, \dots, w_d)}\prod_{j=1}^d(\gamma+n_j-1)!z_j^{n_j}(1-z_j\overline{\varphi_j(w_j)})^{-\gamma-n_j}\nonumber\\
  	&=\overline{a}\prod_{j=1}^d(\gamma+n_j-1)!z_j^{n_j}\overline{w_j}^{n_j}\left(1-\overline{w_j}\varphi_j(0)-z_j\left((1-\overline{w_j}\varphi_j(0))\overline{\varphi_j(0)}+\overline{w_j}\overline{\varphi_j^{\prime}(0)}\right) \right)^{-\gamma-n_j},
 	\end{align} 
 	for all $z=(z_1, \dots, z_d)$, $w=(w_1, \dots, w_d)\in \mathbb{D}^d$.\\
 	 	Since 
 	\begin{align}\label{Them_2_Eqn_16}
 	1-z_j\overline{\varphi_j(0)}-&\overline{w_j}\left((1-z_j\overline{\varphi_j(0)})\varphi_j(0)+z_j\overline{\varphi_j^{\prime}(0)}\right)\nonumber\\
 	&=1-z_j\overline{\varphi_j(0)}-\overline{w_j}\varphi_j(0)-\overline{w_j}z_j\overline{\varphi_j(0)}\varphi_j(0)+\overline{w_j}z_j\overline{\varphi_j^{\prime}(0)}\nonumber\\
 	&=1-\overline{w_j}\varphi_j(0)-z_j\left((1-\overline{w_j}\varphi_j(0))\overline{\varphi_j(0)}+\overline{w_j}\overline{\varphi_j^{\prime}(0)}\right),
 	\end{align}
 	it follows from \eqref{Them_2_Eqn_14} and \eqref{Them_2_Eqn_15} that 
 	$D_{\mfn,\psi, \varphi}^*K_w(z)=D_{\mfn,\psi, \varphi}K_w(z)$. Hence $D_{\mfn,\psi,\varphi}$ is self-adjoint on $\mathcal{H}_\gamma(\mathbb{D}^d)$. This completes the proof.
 \end{proof}
 \begin{remark}
 	If we set $\gamma=1~(\mbox{or}~ 2)$, we obtain both necessary and sufficient condition for the bounded composition-differentiation operator $D_{\mfn,\psi,\varphi}$ to be self-adjoint on the Hardy and Bergman space on the polydisk. More precisely, we have the following corollary which characterizes self-adjoint property of $D_{\mfn,\psi,\varphi}$ on the Berman space over the polydisk.
 \end{remark}
 \begin{corollary}\label{Cor_2.8}
 	Let  $\mfn$ be a non-negative integer. Let $\psi: \mathbb{D}^d\longrightarrow \mathbb{C}$ be a non-zero analytic function and $\varphi: \mathbb{D}^d\longrightarrow \mathbb{D}^d$ with $\varphi(z_1,  \dots, z_d)=(\varphi_1(z_1), \dots, \varphi_d(z_d))$ be an analytic self map of $\mathbb{D}^d$ such that $D_{\mfn,\psi,\varphi}$ is bounded on $\mathcal{L}_a^2(\mathbb{D}^d)$.
 	Then $D_{\mfn,\psi,\varphi}$ is self-adjoint if, and only if,
 	\begin{gather*}
 	\psi(z_1, \dots, z_d)=a\prod_{j=1}^dz_j^{n_j}\big(1-z_j\overline{\varphi_j(0)}\big)^{-n_j-2},
 	\end{gather*}
 	and 
 	\begin{gather*}
 	\varphi_t(z_t)=\varphi_t(0)+\big(1-z_t\overline{\varphi_t(0)}\big)^{-1}z_t\overline{\varphi_t^{\prime}(0)},~ \mbox{for}~ t=1, \dots, d,
 	\end{gather*}
 	where $a=\overline{g(0, \dots, 0)}$.
 \end{corollary}
 \begin{remark}
 	\begin{enumerate}
 		\item[(i)] As a special case of our result we obtain the self-adjoint property for the composition differentiation operator on the Hardy space over the polydisk wich was recenty proved by Ahamed {\it et al.} \cite[Theorem 3.1]{Ahamed_Allu_Rahman_2023}.
 		\item[(ii)] If we put $d=1$ in Theorem \ref{Thm_2}, then we can easily derive self-adjoint property of $D_{\mfn,\psi,\varphi}$  on $\mathcal{H}_\gamma(\mathbb{D})$ on the unit disk (see the work of Lo and Loh \cite[Theorem 4.1]{Lo_Loh_JMAA_2023}).
 		\item[(iii)] From Corollary \ref{Cor_2.8}, we can easily derive self-adjoint property of $D_{\mfn,\psi,\varphi}$ on the Bergman space on the unit disk (see \cite[Theorem 2.4]{Allu_Haldar_Pal_2023}).
 	\end{enumerate} 	   
 \end{remark}

 If we put ${\bf n}=0$, then our composition differentiation operator $D_{\mfn,\psi,\varphi}$ reduces to the weighted composition operator $D_{\psi,\varphi}$. So our next corollary basically says about the self-adjoint property for the composition operator on the polydisk $\mathbb{D}^d$.

 \begin{corollary}\label{Cor_2.10}
	Let $\psi: \mathbb{D}^d\longrightarrow \mathbb{C}$ be a non-zero analytic function and $\varphi: \mathbb{D}^d\longrightarrow \mathbb{D}^d$ with $\varphi(z_1,  \dots, z_d)=(\varphi_1(z_1), \dots, \varphi_d(z_d))$ be an analytic self map of $\mathbb{D}^d$ such that $D_{\psi,\varphi}$ is bounded on $\mathcal{H}_\gamma(\mathbb{D}^d)$.
	Then $D_{\psi,\varphi}$ is self-adjoint if, and only if,
	\begin{gather*}
	\psi(z_1, \dots, z_d)=a\prod_{j=1}^d\big(1-z_j\overline{\varphi_j(0)}\big)^{-n_j-\gamma},
	\end{gather*}
	and 
	\begin{gather*}
	\varphi_t(z_t)=\varphi_t(0)+\big(1-z_t\overline{\varphi_t(0)}\big)^{-1}z_t\overline{\varphi_t^{\prime}(0)},~ \mbox{for}~ t=1, \dots, d,
	\end{gather*}
	where $a=\overline{g(0, \dots, 0)}$.
\end{corollary}
\begin{remark}
	\begin{enumerate}
		\item[(i)] By setting $d=1$ ({\it i.e.,} unit disk case) in Corollary \ref{Cor_2.10}, we obtain the self-adjoint property for the weighted composition operator $D_{\psi,\varphi}$ on $\mathcal{H}_\gamma(\mathbb{D})$, (see \cite[Theorem 2.9]{Lim_Khoi_JMAA_2018}).
		\item[(ii)]  By setting $\gamma=1~ \mbox{or}~ 2$ and $d=1$ in Corollary \ref{Cor_2.10}, we obtain the self-adjoint property for the weighted composition operator $D_{\psi,\varphi}$ on the Hardy space (see \cite[Theorem 3.3]{Han_Wang_JMAA_2021}) and for the Bergman space (see \cite[Theorem 2.3]{Liu_Ponnusamy_Xie_LAMA_2023}). 
	\end{enumerate}
\end{remark}

\subsection{ Generalized complex symmetric composition differentiation operators on $\mathcal{H}_\gamma(\mathbb{D})$}
In this subsection, we discuss complex symmetry, self-adjoint property of the generalized weighted composition differentiation operators on $\mathcal{H}_\gamma(\mathbb{D})$.
We denote $M_{n, \psi, \varphi}$ by generalized weighted composition differentiation  operators on $\mathcal{H}_\gamma(\mathbb{D})$ which is defined by 
$$M_{n, \psi, \varphi}=\displaystyle\sum_{j=1}^{n}a_jD_{j,\psi_j, \varphi},$$
where $a_j\in \mathbb{C}$ for $j=1, 2, \dots, n$.
\\ \\
Let $\mathcal{H}(\mathbb{D})$ be the space of all analytic functions in the unit disk $\mathbb{D}$.
For $\mu, \xi$ on the unit circle $z\in \mathbb{C}: |z|=1$, the conjugation $C_{\mu, \xi}$ is defined as $C_{\mu, \xi}f=\mu \overline{f(\overline{\xi z})}$, where $f$ belongs to the space of analytic functions. \\
{\bf Question-$2$: Is it possible to study the complex symmetric structure of the generalized weighted composition operators $M_{n, \psi, \varphi}$ on $\mathcal{H}_\gamma(\mathbb{D})$ concerning conjugations $C_{\mu, \xi}$?}\\

The following very basic lemma is essential for our analysis.
\begin{lemma}
 Let $\varphi$ be an analytic self-map of $\mathbb{D}$ and $\psi_j\in \mathcal{H}(\mathbb{D})~(j=1, 2,\dots, n)$ such that $M_{n, \psi, \varphi}=\displaystyle\sum_{j=1}^{n}a_jD_{j,\psi_j, \varphi} $ (where $a_j\in \mathbb{C}$ for $j=1, 2, \dots, n$) is bounded on $\mathcal{H}_\gamma(\mathbb{D})$.
  	Then 
  	$$ M_{n,\psi, \varphi}^*K_w=\sum_{j=1}^{n}\overline{a_j\psi_j(w)}K_{\varphi(w)}^{[j]}.$$
\end{lemma}

\begin{proof}
	Let $f\in \mathcal{H}_\gamma(\mathbb{D})$, then we have
	\begin{align*}
	\langle f,  M_{n,\psi, \varphi}^*K_w \rangle_{\mathcal{H}_\gamma(\mathbb{D})}&=  \langle  M_{n,\psi, \varphi}f, K_w \rangle_{\mathcal{H}_\gamma(\mathbb{D})}\\
	&=\bigg\langle \sum_{j=1}^{n}a_jD_{j,\psi_j, \varphi} , K_w \bigg\rangle_{\mathcal{H}_\gamma(\mathbb{D})}\\
	&=\sum_{j=1}^{n}\bigg\langle a_j\psi_j(f^{(j)}o \varphi) , K_w \bigg\rangle_{\mathcal{H}_\gamma(\mathbb{D})}\\
	&=\sum_{j=1}^{n}a_j\psi_j(w) f^{(j)}(\varphi(w)) \\
	&=\sum_{j=1}^{n}\bigg\langle  f, \overline{a_j\psi_j(w)}K_{\varphi(w)}^{[j]} \bigg\rangle_{\mathcal{H}_\gamma(\mathbb{D})}\\
		&=\bigg\langle  f, \sum_{j=1}^{n}\overline{a_j\psi_j(w)}K_{\varphi(w)}^{[j]} \bigg\rangle_{\mathcal{H}_\gamma(\mathbb{D})}.
	\end{align*}
	So, we have
	$$\bigg\langle f,  M_{n,\psi, \varphi}^*K_w \bigg\rangle_{\mathcal{H}_\gamma(\mathbb{D})}=\bigg\langle  f,  \sum_{j=1}^{n}\overline{a_j\psi_j(w)}K_{\varphi(w)}^{[j]} \bigg\rangle_{\mathcal{H}_\gamma(\mathbb{D})}~~\mbox{for~ all}~ f\in \mathcal{H}_\gamma(\mathbb{D}),$$
	which leads to the following
$$ M_{n,\psi, \varphi}^*K_w=\sum_{j=1}^{n}\overline{a_j\psi_j(w)}K_{\varphi(w)}^{[j]}.$$
\end{proof}
In the following theorem, we provide a necessary and sufficient condition for which of the generalized weighted  composition differentiation operator $M_{n,\psi, \varphi} $ is complex symmetric on $\mathcal{H}_\gamma(\mathbb{D})$ with the conjugation $C_{\mu, \xi}$.

\begin{theorem}\label{Thm_3}
	  Let $\varphi: \mathbb{D}\longrightarrow \mathbb{D}$ be an analytic self-map of $\mathbb{D}$ and $\psi_j\in \mathcal{H}(\mathbb{D})~(j=1, 2,\dots, n)$ with $\psi_j \neq 0$ such that $M_{n, \psi, \varphi}=\displaystyle\sum_{j=1}^{n}a_jD_{j,\psi_j, \varphi} $ (where $a_j\in \mathbb{C}$ for $j=1, 2, \dots, n$) is bounded on $\mathcal{H}_\gamma(\mathbb{D})$. Then $M_{n,\psi,\varphi}$ is complex symmetric with respect to the conjugation $\mathcal{C}_{\mu, \xi}$ if, and only if,
	\begin{gather*}
	\psi_j(z)=c_jz^{j}\big(1-\xi \varphi(0) z\big)^{-\gamma-j},
	\end{gather*}
	and 
	\begin{gather*}
	\varphi(z)=\varphi(0)+\big(1-\xi \varphi(0) z\big)^{-1}\varphi^{\prime}(0)z,~ \mbox{for}~ z\in \mathbb{D},
	\end{gather*}
	where $c_j\in \mathbb{C}$, $j=1, \dots, n$.
\end{theorem}
\begin{proof}
Let $M_{n, \psi, \varphi}$ be complex symmetric with respect to the conjugation $\mathcal{C}_{\mu, \xi}$. 
Then we have
\begin{align}\label{Them_Eqn_5_2}
M_{n,\psi, \varphi}\mathcal{C}_{\mu, \xi}K_w(z)= \mathcal{C}_{\mu, \xi}M_{n,\psi, \varphi}^*K_w(z)
\end{align}
for all $z,w\in \mathbb{D}$.
A simple computation in \eqref{Them_Eqn_5_2} gives
 \begin{align}\label{Them_Eqn_5_3}
M_{n,\psi, \varphi}\mathcal{C}_{\mu, \xi}K_w(z)&=M_{n,\psi, \varphi}\mathcal{C}_{\mu, \xi}\left(\frac{1}{(1-\overline{w}z)^{\gamma}}\right)\nonumber\\
&=M_{n,\psi, \varphi}\left(\frac{\mu}{(1-\xi wz)^{\gamma}}\right)\nonumber\\
&=\displaystyle\sum_{j=1}^{n}a_jD_{j,\psi_j, \varphi} \left(\frac{\mu}{(1-\xi wz)^{\gamma}}\right)\nonumber\\
&=\displaystyle\sum_{j=1}^{n}a_j(\gamma+j-1)!\mu \psi_j(z)(\xi w)^{j}(1-\xi w\varphi(z))^{-\gamma-j}.
\end{align}
and 

\begin{align}\label{Them_Eqn_5_4}
\mathcal{C}_{\mu, \xi} M_{n,\psi, \varphi}^*K_w(z)&=\mathcal{C}_{\mu, \xi}\displaystyle\sum_{j=1}^{n}\overline{a_j\psi_j(w)}K_{\varphi(w)}^{[j]}(z)\nonumber\\
&=\mathcal{C}_{\mu, \xi}\displaystyle\sum_{j=1}^{n}\overline{a_j\psi_j(w)}(\gamma+j-1)!z^{j}(1-\overline{\varphi(w)}z)^{-\gamma-j}\nonumber\\
&=\displaystyle\sum_{j=1}^{n}a_j(\gamma+j-1)!\mu\psi_j(w)(\xi z)^{j}(1-\varphi(w)\xi z)^{-\gamma-j}.
\end{align}
By using \eqref{Them_Eqn_5_3} and \eqref{Them_Eqn_5_4}, we obtain
\begin{align*}
\displaystyle\sum_{j=1}^{n}a_j(\gamma+j-1)!\mu \psi_j(z)(\xi w)^{j}&(1-\xi w\varphi(z))^{-\gamma-j}\\
&=\displaystyle\sum_{j=1}^{n}a_j(\gamma+j-1)!\mu\psi_j(w)(\xi z)^{j}(1-\varphi(w)\xi z)^{-\gamma-j}.
\end{align*}
Since $\mu, \xi\in \{z\in \mathbb{C}: |z|=1\}$, the above equation implies that 

\begin{equation}\label{Them_Eqn_5_5}
 \psi_j(z)w^{j}(1-\xi w\varphi(z))^{-\gamma-j}=\psi_j(w)z^{j}(1-\xi z\varphi(w))^{-\gamma-j}
\end{equation}
for all $z,w\in \mathbb{D}$ and $j=1,2, \dots, n$.\\

By setting $z=0$ in \eqref{Them_Eqn_5_5}, we get $\psi_j(0)=0$ for  $j=1,2, \dots, n$. Let $\psi_j(z)=z^mg_j(z)$, where $m\in \mathbb{N}$ and $g_j$ is analytic on $\mathbb{D}$ with $g_j(0)\neq 0$. Now our aim is to show that $m=j$.\\

 \underline{\bf Case-I:} If $m>j$, from \eqref{Them_Eqn_5_5}, it follows that
 
\begin{equation}\label{Them_Eqn_5_6}
z^{m-j}g_j(z)(1-\xi w\varphi(z))^{-\gamma-j}=w^{m-j}g_j(w)(1-\xi z\varphi(w))^{-\gamma-j}
\end{equation}
for all $z,w\in \mathbb{D}$ and $j=1,2, \dots, n$.
Putting $w=0$ in \eqref{Them_Eqn_5_6}, we get $g_j(z)=0$ on $\mathbb{D}$, which is a contradiction of the fact that $g_j(0)\neq 0$. \\

 \underline{\bf Case-II:} If $m<j$, from \eqref{Them_Eqn_5_5}, it follows that

\begin{equation}\label{Them_Eqn_5_7}
w^{j-m}g_j(z)(1-\xi w\varphi(z))^{-\gamma-j}=z^{j-m}g_j(w)(1-\xi z\varphi(w))^{-\gamma-j}
\end{equation}
for all $z,w\in \mathbb{D}$ and $j=1,2, \dots, n$. By setting $w=0$ in \eqref{Them_Eqn_5_7}, we have $g_j(0)=0$, which contradicts the fact that $g_j(0)\neq 0$. Therefore, we must have $m=j$ and now \eqref{Them_Eqn_5_5} is equivalent to 

  \begin{equation}\label{Them_Eqn_5_8}
  g_j(z)(1-\xi w\varphi(z))^{-\gamma-j}=g_j(w)(1-\xi z\varphi(w))^{-\gamma-j}
  \end{equation}
for all $z,w\in \mathbb{D}$.
 By setting $w=0$ in \eqref{Them_Eqn_5_8}, we obtain

 \begin{equation}\label{Them_Eqn_5_9}
g_j(z)=g_j(0)(1-\xi z\varphi(0))^{-\gamma-j}
\end{equation}
for $j=1,2, \dots, n$. Thus, we have
 \begin{align*}
\psi_j(z)&=z^jg_j(z)\\
&=z^jg_j(0)(1-\xi z\varphi(0))^{-\gamma-j}\\
&=c_jz^j(1-\xi z\varphi(0))^{-\gamma-j},
\end{align*}
where $c_j=g_j(0), j=1, \dots, n$.

By substituting the value of $\psi_j(z)$ in $\eqref{Them_Eqn_5_5}$, we obtain

\begin{equation}\label{Them_Eqn_5_10}
(1-\xi z\varphi(0))^{\gamma+j}(1-\xi w\varphi(z))^{\gamma+j}=(1-\xi w\varphi(0))^{\gamma+j}(1-\xi z\varphi(w))^{\gamma+j}
\end{equation}
for all $z,w\in \mathbb{D}$. By differentiating both sides of \eqref{Them_Eqn_5_10} with respect to $w$, we obtain
\begin{align}\label{Eq_3.10}
(\gamma+j)(1-\xi z\varphi(0))^{\gamma+j}&(1-\xi w\varphi(z))^{\gamma+j-1}(-\xi \varphi(z))\nonumber\\
&=(\gamma+j)(1-\xi w\varphi(0))^{\gamma+j-1}(1-\xi z\varphi(w))^{\gamma+j}(-\xi \varphi(0))\nonumber\\
&+(\gamma+j)(1-\xi w\varphi(0))^{\gamma+j}(1-\xi z\varphi(w))^{\gamma+j-1}(-\xi z\varphi^\prime(w)).
\end{align}

By substituting $w=0$ in \eqref{Eq_3.10}, we obtain
\begin{align*}
(\gamma+j)(1-\xi z\varphi(0))^{\gamma+j}(-\xi \varphi(z))
&=(\gamma+j)(1-\xi z\varphi(0))^{\gamma+j}(-\xi \varphi(0))\\
&+(\gamma+j)(1-\xi z\varphi(0))^{\gamma+j-1}(-\xi z\varphi^\prime(0)).
\end{align*} 
Thus,
\begin{align*}
	\varphi(z)&=\varphi(0)+\frac{z\varphi^\prime(0)}{1-\xi z\varphi(0)}.	
\end{align*}

Conversly, let $c_1, c_2, \dots, c_n \in \mathbb{C}$ be such that  
\begin{gather*}
\psi_j(z)=c_jz^{j}\big(1-\xi \varphi(0) z\big)^{-\gamma-j},
\end{gather*}
for $j=1, 2, \dots, n$, and
\begin{align*}
\varphi(z)&=\varphi(0)+\frac{z\varphi^\prime(0)}{1-\xi z\varphi(0)}.	
\end{align*}
for all $z\in \mathbb{D}$. Now, using \eqref{Them_Eqn_5_3} and \eqref{Them_Eqn_5_4}, we obtain

\begin{align}\label{Them_Eqn_5_11}
&M_{n,\psi, \varphi}\mathcal{C}_{\mu, \xi}K_w(z)
=\displaystyle\sum_{j=1}^{n}a_j(\gamma+j-1)!\mu \psi_j(z)(\xi w)^{j}(1-\xi w\varphi(z))^{-\gamma-j}\nonumber\\
&=\displaystyle\sum_{j=1}^{n}a_j(\gamma+j-1)!\mu c_jz^j(\xi w)^{j}(1-\xi w\varphi(z))^{-\gamma-j}(1-\xi z\varphi(0))^{-\gamma-j}\nonumber\\
&=\displaystyle\sum_{j=1}^{n}a_j(\gamma+j-1)!\mu c_j(\xi wz)^{j}\left(1-\xi w\left(\varphi(0)+\frac{z\varphi^\prime(0)}{1-\xi z\varphi(0)}\right)\right)^{-\gamma-j}(1-\xi z\varphi(0))^{-\gamma-j}\nonumber\\
&=\displaystyle\sum_{j=1}^{n}a_j(\gamma+j-1)!\mu c_j(\xi wz)^{j}\left(1-\xi z\varphi(0)-\xi w\varphi(0)+\xi^2wz\varphi^2(0)-\xi wz\varphi^\prime(0)\right)^{-\gamma-j}.
\end{align}

Similarly, we have

\begin{align}\label{Them_Eqn_5_12}
&\mathcal{C}_{\mu, \xi} M_{n,\psi, \varphi}^*K_w(z)
=\displaystyle\sum_{j=1}^{n}a_j(\gamma+j-1)!\mu \psi_j(w)(\xi z)^{j}(1-\xi z\varphi(w))^{-\gamma-j}\nonumber\\
&=\displaystyle\sum_{j=1}^{n}a_j(\gamma+j-1)!\mu c_jw^j(\xi z)^{j}(1-\xi z\varphi(w))^{-\gamma-j}(1-\xi w\varphi(0))^{-\gamma-j}\nonumber\\
&=\displaystyle\sum_{j=1}^{n}a_j(\gamma+j-1)!\mu c_j(\xi wz)^{j}\left(1-\xi z\left(\varphi(0)+\frac{w\varphi^\prime(0)}{1-\xi w\varphi(0)}\right)\right)^{-\gamma-j}(1-\xi w\varphi(0))^{-\gamma-j}\nonumber\\
&=\displaystyle\sum_{j=1}^{n}a_j(\gamma+j-1)!\mu c_j(\xi wz)^{j}\left(1-\xi w\varphi(0)-\xi z\varphi(0)+\xi^2wz\varphi^2(0)-\xi wz\varphi^\prime(0)\right)^{-\gamma-j}.
\end{align}
It follows from \eqref{Them_Eqn_5_11} and \eqref{Them_Eqn_5_12} that $M_{n,\psi, \varphi}\mathcal{C}_{\mu, \xi}K_w(z)
=\mathcal{C}_{\mu, \xi} M_{n,\psi, \varphi}^*K_w(z)$ for all $z\in \mathbb{D}$.
Since the span of the reproducing kernel functions is dense in $\mathcal{H}_\gamma(\mathbb{D})$, the operator $ M_{n,\psi, \varphi}$ is complex symmetric with conjugation $\mathcal{C}_{\mu, \xi}$. This complete the proof.

\end{proof}

\begin{remark}
	If we set $\gamma=1 (\mbox{or}~ 2)$, we obtain explicit forms of $\psi$ and $\varphi$ for which $M_{n,\psi,\varphi}$ is complex symmetric with respect to the conjugation $\mathcal{C}_{\mu, \xi}$ on the Hardy and Bergman space. More precisely, we have the following corollary, which deals with the complex symmetry of $M_{n,\psi,\varphi}$  with respect to the conjugation $\mathcal{C}_{\mu, \xi}$ on the Bergman space.
\end{remark}

\begin{corollary}
	 Let $\varphi: \mathbb{D}\longrightarrow \mathbb{D}$ be an analytic self-map of $\mathbb{D}$ and $\psi_j\in \mathcal{H}(\mathbb{D})~(j=1, 2,\dots, n)$ with $\psi_j \neq 0$ such that $M_{n, \psi, \varphi}=\displaystyle\sum_{j=1}^{n}a_jD_{j,\psi_j, \varphi} $ (where $a_j\in \mathbb{C}$ for $j=1, 2, \dots, n$) is bounded on $\mathcal{L}_a^2(\mathbb{D})$. Then $M_{n,\psi,\varphi}$ is complex symmetric with respect to the conjugation $\mathcal{C}_{\mu, \xi}$ if, and only if,
	\begin{gather*}
	\psi_j(z)=c_jz^{j}\big(1-\xi \varphi(0) z\big)^{-j-2},
	\end{gather*}
	and 
	\begin{gather*}
	\varphi(z)=\varphi(0)+\big(1-\xi \varphi(0) z\big)^{-1}\varphi^{\prime}(0)z,~ \mbox{for}~ z\in \mathbb{D},
	\end{gather*}
	where $c_j\in \mathbb{C}$, $j=1, \dots, n$.
\end{corollary}

\begin{remark}
	\begin{enumerate}
		\item[(i)] 	As a special case of our result ({\it i.e.,} if we set $\gamma=\alpha+2$), we obtain the complex symmetric structure of the generalized composition differentiation operator on the weighted Bergman space on the unit disk, which was recently proved by Ahamed and Rahman \cite[Theorem 2.1]{Ahamed_Rahman_2023}.
	\item[(ii)] It is not difficult to see that, if $a_1=a_2=\dots=a_{n-1}=0$ and $a_n=1$, then $M_{n,\psi,\varphi}$ becomes $D_{n,\psi,\varphi}$. Moreover we have the following result which can be easily obtained.
	\end{enumerate}
	
\end{remark}

\begin{corollary}\label{Cor_3.6}
 Let $\varphi: \mathbb{D}\longrightarrow \mathbb{D}$ be an analytic self-map of $\mathbb{D}$ and $\psi\in \mathcal{H}(\mathbb{D})$ with $\psi \neq 0$ such that $D_{n, \psi, \varphi}$ is bounded on $\mathcal{H}_\gamma(\mathbb{D})$. Then $D_{n,\psi,\varphi}$ is complex symmetric with respect to the conjugation $\mathcal{C}_{\mu, \xi}$ if, and only if,
\begin{gather*}
\psi(z)=cz^{n}\big(1-\xi \varphi(0) z\big)^{-\gamma-n},
\end{gather*}
and 
\begin{gather*}
\varphi(z)=\varphi(0)+\big(1-\xi \varphi(0) z\big)^{-1}\varphi^{\prime}(0)z,~ \mbox{for}~ z\in \mathbb{D},
\end{gather*}
where $c\in \mathbb{C}$.
\end{corollary}

\begin{remark}
	\begin{enumerate}
		\item [(i)] Complex symmetric structure of the weighted composition differentiation operator $D_{n,\psi,\varphi}$ on $\mathcal{H}_\gamma(\mathbb{D})$ with respect to other conjuations already been established in \cite{Lo_Loh_JMAA_2023}.
	\item[(ii)] Similarly, if we set $\gamma=1(\mbox{or}~ 2 ~\mbox{or}~ \alpha+2)$ in Corollary \ref{Cor_3.6}, we obtain explicit forms of $\psi$ and $\varphi$ for which $D_{n,\psi,\varphi}$ is complex symmetric with respect to the conjugation $\mathcal{C}_{\mu, \xi}$ on the Hardy and Bergman space, see \cite{Allu_Haldar_Pal_2023, Han_Wang_JMAA_2021,Han_Wang_CAOT_2021} for more details. 
\item[(iii)]  If we put ${n}=0$, then our composition differentiation operator $D_{n,\psi,\varphi}$ reduces to the weighted composition operfator $D_{\psi,\varphi}$. So in a similar way, we can obtain the complex symmetric structure of the composition operator on the unit disk $\mathbb{D}$, which has been proved by Lim and Khoi \cite{Lim_Khoi_JMAA_2018}.
	\end{enumerate}	
\end{remark}

 In Theorem \ref{Thm_4}, we obtain a condition that is both necessary and sufficient for the  bounded generalized composition differentiation operator $M_{n,\psi,\varphi}$ to satisfy self-adjoint properties on  $\mathcal{H}_\gamma(\mathbb{D})$.

\begin{theorem}\label{Thm_4}
	Let $\varphi: \mathbb{D}\longrightarrow \mathbb{D}$ be an analytic self-map of $\mathbb{D}$ and $\psi_j (j=1, 2, \dots, n)$ be non-zero analytic functions on the unit disk such that $M_{n, \psi, \varphi}=\displaystyle\sum_{j=1}^{n}a_jD_{j,\psi_j, \varphi} $ is bounded on $\mathcal{H}_\gamma(\mathbb{D})$ and $a_j\in \mathbb{R}$ for $j=1, 2, \dots, n$. Then $M_{n,\psi,\varphi}$ is Hermitian if, and only if,
	\begin{gather*}
	\psi_j(z)=c_jz^{j}\big(1- \overline{\varphi(0)} z\big)^{-\gamma-j},
	\end{gather*}
	and 
	\begin{gather*}
	\varphi(z)=\varphi(0)+\big(1- \overline{\varphi(0)} z\big)^{-1}\overline{\varphi^{\prime}(0)}z,~ \mbox{for}~ z\in \mathbb{D},
	\end{gather*}
	where $c_j, \varphi^\prime(0)\in \mathbb{R}$, $j=1, \dots, n$.
\end{theorem}

\begin{proof}
		Let $M_{n, \psi, \varphi}$ be Hermitian. 
	Then we have
	\begin{align}\label{Them_4_Eqn_5_2}
	M_{n,\psi, \varphi}^*K_w(z)= M_{n,\psi, \varphi}K_w(z)
	\end{align}
	for all $z,w\in \mathbb{D}$.
	A simplification in \eqref{Them_4_Eqn_5_2} gives
	\begin{align}\label{Them_4_Eqn_5_3}
	M_{n,\psi, \varphi}K_w(z)&=\displaystyle\sum_{j=1}^{n}a_jD_{j,\psi_j, \varphi} K_w(z)\nonumber\\
	&=\displaystyle\sum_{j=1}^{n}a_j(\gamma+j-1)! \psi_j(z)\overline{w}^{j}(1-\overline{w}\varphi(z))^{-\gamma-j}.
	\end{align}

	and 
	
	\begin{align}\label{Them_4_Eqn_5_4}
	 M_{n,\psi, \varphi}^*K_w(z)&=\displaystyle\sum_{j=1}^{n}\overline{a_j\psi_j(w)}K_{\varphi(w)}^{[j]}(z)\nonumber\\
	&=\displaystyle\sum_{j=1}^{n}\overline{a_j\psi_j(w)}(\gamma+j-1)!z^{j}(1-\overline{\varphi(w)}z)^{-\gamma-j}.
	\end{align}
	By using \eqref{Them_4_Eqn_5_3} and \eqref{Them_4_Eqn_5_4} in \eqref{Them_4_Eqn_5_2}, we obtain 
	\begin{align*}
	\displaystyle\sum_{j=1}^{n}\overline{a_j \psi_j(w)}(\gamma+j-1)!z^{j}&(1- \overline{\varphi(w)}z)^{-\gamma-j}\\
	&=\displaystyle\sum_{j=1}^{n}a_j\psi_j(z)(\gamma+j-1)!\overline{w}^{j}(1-\varphi(z)\overline{w})^{-\gamma-j},
	\end{align*}
which implies that 
	
	\begin{equation}\label{Them_4_Eqn_5_5}
	\overline{ \psi_j(w)}z^{j}(1-\overline{\varphi(w)}z)^{-\gamma-j}=\psi_j(z)\overline{w}^{j}(1-\varphi(z)\overline{w})^{-\gamma-j}
	\end{equation}
	for all $z,w\in \mathbb{D}$ and $j=1,2, \dots, n$.	
	By setting $w=0$ in \eqref{Them_4_Eqn_5_5}, we get $\psi_j(0)=0$ for  $j=1,2, \dots, n$. Let $\psi_j(z)=z^mg_j(z)$, where $m\in \mathbb{N}$ and $g_j$ is analytic on $\mathbb{D}$ with $g_j(0)\neq 0$. Now our aim is to show that $m=j$.
	\\ \\
	\underline{\bf Case-I:} If $m>j$, from \eqref{Them_4_Eqn_5_5}, it follows that
	
	\begin{equation}\label{Them_4_Eqn_5_6}
	z^{m-j}g_j(z)(1- \overline{w}\varphi(z))^{-\gamma-j}=\overline{w}^{m-j}\overline{g_j(w)}(1- z\overline{\varphi(w)})^{-\gamma-j}
	\end{equation}
	for all $z,w\in \mathbb{D}$ and $j=1,2, \dots, n$.
	Putting $w=0$ in \eqref{Them_4_Eqn_5_6}, we obtain $g_j(z)=0$ on $\mathbb{D}$, which contradicts to the fact that $g_j(0)\neq 0$. \\ \\	
		\underline{\bf Case-II:} If $m<j$, from \eqref{Them_4_Eqn_5_5}, it follows that
	
	\begin{equation}\label{Them_4_Eqn_5_7}
	\overline{w}^{j-m}g_j(z)(1-\overline{w}\varphi(z))^{-\gamma-j}=z^{j-m}\overline{g_j(w)}(1-z\overline{\varphi(w)})^{-\gamma-j}
	\end{equation}
	for all $z,w\in \mathbb{D}$ and $j=1,2, \dots, n$. By setting $w=0$ in \eqref{Them_4_Eqn_5_7}, we obtain $g_j(0)=0$, which contradicts the fact that $g_j(0)\neq 0$. Therefore, we must have $m=j$ and now \eqref{Them_4_Eqn_5_5} is reduces to 
	
	\begin{equation}\label{Them_4_Eqn_5_8}
	g_j(z)(1- \overline{w}\varphi(z))^{-\gamma-j}=\overline{g_j(w)}(1- z\overline{\varphi(w)})^{-\gamma-j}
	\end{equation}
	for all $z,w\in \mathbb{D}$.\\
	
	By setting $w=0$ in \eqref{Them_4_Eqn_5_8}, we obtain
	
	\begin{equation}\label{Them_4_Eqn_5_9}
	g_j(z)=\overline{g_j(0)}(1- z\overline{\varphi(0)})^{-\gamma-j}
	\end{equation}
	for $j=1,2, \dots, n$.	
	Thus, we obtain
	\begin{align*}
	\psi_j(z)&=z^jg_j(z)\\
	&=z^j\overline{g_j(0)}(1-z\overline{\varphi(0)})^{-\gamma-j}\\
	&=c_jz^j(1- z\overline{\varphi(0)})^{-\gamma-j},
	\end{align*}
	where $c_j=\overline{g_j(0)}, j=1, \dots, n$.	
	By substituting the value of $\psi_j(z)$ in $\eqref{Them_4_Eqn_5_5}$, we obtain
	
	\begin{equation}\label{Them_4_Eqn_5_10}
	\overline{c_j}(1- z\overline{\varphi(0)})^{\gamma+j}(1-\overline{ w}\varphi(z))^{\gamma+j}=c_j(1- \overline{w}\varphi(0))^{\gamma+j}(1- z\overline{\varphi(w)})^{\gamma+j}
	\end{equation}
	for all $z,w\in \mathbb{D}$.
	If we put $w=0$ in \eqref{Them_4_Eqn_5_10}, then it is not very difficult to see that $c_j=\overline{c_j}$ and hence $c_j\in \mathbb{R}$ for all $j=1, 2, \dots, n.$
	Differentiating both sides of \eqref{Them_4_Eqn_5_10} with respect to $\overline{w}$, we obtain
	\begin{align}\label{EQ_3.22}
	(\gamma+j)(1- z\overline{\varphi(0)})^{\gamma+j}&(1-\overline{ w}\varphi(z))^{\gamma+j-1}(\varphi(z))\nonumber\\
	&=(\gamma+j)(1-\overline{w}\varphi(0))^{\gamma+j-1}(1- z\overline{\varphi(w)})^{\gamma+j}(\varphi(0))\nonumber\\
	&+(\gamma+j)(1-\overline{w}\varphi(0))^{\gamma+j}(1- z\overline{\varphi(w)})^{\gamma+j-1}(z\overline{\varphi^\prime(w)}).
	\end{align}
	
	Putting $w=0$ in \eqref{EQ_3.22}, we obtain
	\begin{align*}
	(1-z\overline{\varphi(0)})^{\gamma+j}\varphi(z)
=(1- z\overline{\varphi(0)})^{\gamma+j}\times\varphi(0)
	+(1-z\overline{\varphi(0)})^{\gamma+j-1}\times(z\overline{\varphi^\prime(0)}).
	\end{align*} 
	Thus,
	\begin{align}\label{Them_4_Eqn_5_13}
	\varphi(z)&=\varphi(0)+\frac{z\overline{\varphi^\prime(0)}}{1-z\overline{\varphi(0)}}.	
	\end{align}
	Differentiating \eqref{Them_4_Eqn_5_13} with respect to $z$, we obtain
			\begin{align}\label{Them_4_Eqn_5_14}
	\varphi^\prime(z)&=\frac{\overline{\varphi^\prime(0)}}{(1-z\overline{\varphi(0)})^2}.	
	\end{align}
	Thus, $	\varphi^\prime(0)=\overline{\varphi^\prime(0)}$, which implies $\varphi^\prime(0)$ is a real number.\\

	Conversly, let us assume that  
	\begin{gather*}
	\psi_j(z)=c_jz^{j}\big(1-\overline{\varphi(0)} z\big)^{-\gamma-j},
	\end{gather*}
	where $c_1, c_2, \dots, c_n \in \mathbb{R}$ and for $j=1, 2, \dots, n$,

	and
	\begin{align*}
	\varphi(z)&=\varphi(0)+\frac{z\overline{\varphi^\prime(0)}}{1-z\overline{\varphi(0)}}.	
	\end{align*}

	Now, using \eqref{Them_4_Eqn_5_3} and \eqref{Them_4_Eqn_5_4}, we obtain
	\begin{align}\label{Them_4_Eqn_5_11}
	&M_{n,\psi, \varphi}K_w(z)
	=\displaystyle\sum_{j=1}^{n}a_j(\gamma+j-1)! \psi_j(z)\overline{w}^{j}(1- \overline{w}\varphi(z))^{-\gamma-j}\nonumber\\
		&=\displaystyle\sum_{j=1}^{n}a_j(\gamma+j-1)! c_j(z\overline{w})^{j}\left(1-\overline{w}\left(\varphi(0)+\frac{z\overline{\varphi^\prime(0)}}{1-z\overline{\varphi(0)}}\right)\right)^{-\gamma-j}(1- z\overline{\varphi(0)})^{-\gamma-j}\nonumber\\
	&=\displaystyle\sum_{j=1}^{n}a_jc_j(\gamma+j-1)! (z\overline{w})^{j}\left(1-\overline{\varphi(0)}z-\varphi(0)\overline{w}+\overline{w}\varphi(0)\overline{\varphi(0)}z-\overline{w}z\overline{\varphi^\prime(0)}\right)^{-\gamma-j}.
	\end{align}

	Similarly,
		
	\begin{align}\label{Them_4_Eqn_5_12}
	& M_{n,\psi, \varphi}^*K_w(z)
	=\displaystyle\sum_{j=1}^{n}\overline{a_j\psi_j(w)}(\gamma+j-1)! z^{j}(1-z\overline{\varphi(w)})^{-\gamma-j}\nonumber\\
	&=\displaystyle\sum_{j=1}^{n}\overline{a_jc_j}(\gamma+j-1)!\overline{w}^jz^{j}\left(1-z\left(\overline{\varphi(0)}+\frac{\overline{w}\varphi^\prime(0)}{1-\overline{w}\varphi(0)}\right)\right)^{-\gamma-j}(1- w\overline{\varphi(0)})^{-\gamma-j}\nonumber\\
	&=\displaystyle\sum_{j=1}^{n}\overline{a_jc_j}(\gamma+j-1)!(\overline{w}z)^{j}\left(1-\overline{w}\varphi(0)-z\overline{\varphi(0)}+z\overline{\varphi(0)}\varphi(0)\overline{w}-z\overline{w}\varphi^\prime(0)\right)^{-\gamma-j}.
	\end{align}

Since $a_j, c_j$ and $\varphi^\prime(0)$ are reals, hence it follows from \eqref{Them_4_Eqn_5_11} and \eqref{Them_4_Eqn_5_12} that $M_{n,\psi, \varphi}K_w(z)
	= M_{n,\psi, \varphi}^*K_w(z)$ for all $z\in \mathbb{D}$. Since the span of the reproducing kernel functions is dense in $\mathcal{H}_\gamma(\mathbb{D})$, so $M_{n,\psi, \varphi}^*= M_{n,\psi, \varphi}$ which gives that $M_{n,\psi, \varphi}$ is Hermitian. This completes the proof.
\begin{remark}
		\begin{enumerate}
			\item[(i)] 	As a special case of our result ({\it i.e.,} if we set $\gamma=\alpha+2$), we obtain the Hermitian property of the generalized composition differentiation operator on the weighted Bergman space on the unit disk, which was recently proved by Ahamed and Rahman \cite[Theorem 2.3]{Ahamed_Rahman_2023}.
			\item[(ii)] It is not difficult to see that, if $a_1=a_2=\dots=a_{n-1}=0$ and $a_n=1$, then $M_{n,\psi,\varphi}$ reduces $D_{n,\psi,\varphi}$. Moreover, we obtain the following result which can be easily derived.
		\end{enumerate}
		\end{remark}
	
	\begin{corollary}\label{Cor_3.6_1}\cite[Theorem 4.1]{Lo_Loh_JMAA_2023}
		Let $\varphi: \mathbb{D}\longrightarrow \mathbb{D}$ and $\psi:\mathbb{D}\longrightarrow \mathbb{C} $ be an analytic functions. Then $D_{n,\psi,\varphi}:\mathcal{H}_\gamma\rightarrow \mathcal{H}_\gamma$ is Hermitian if, and only if,
		\begin{gather*}
		\psi(z)=cz^{n}\big(1-\overline{\varphi(0)} z\big)^{-\gamma-n},
		\end{gather*}
		and 
		\begin{gather*}
		\varphi(z)=\varphi(0)+\big(1-\overline{\varphi(0)} z\big)^{-1}\overline{\varphi^{\prime}(0)}z,~ \mbox{for}~ z\in \mathbb{D},
		\end{gather*}
where	$c, \varphi^\prime(0)\in \mathbb{R}$.
	\end{corollary}

\end{proof}

\section{Composition operators and their Berezin range}

Recall that ${\mathcal{L}}(\mathcal{H})$ be the $C^*$-algebra of all bounded linear operators on a complex Hilbert space $\mathcal{H}$ with inner product $\langle \cdot,\cdot\rangle$ and the corresponding norm $\|.\|$. 
The {\it numerical range} of $S\in
\mathcal{L(H)}$ is defined as
$ W(S)=\{\langle Sx, x \rangle: x\in \mathcal{H}, \|x\|=1
\}.$ 
The {\it numerical radius} of $S$ is
defined as $ w(S)=\displaystyle\sup\{|z|: z\in W(S) \}.$
The numerical radius
$w(\cdot)$ defines a norm on $\mathcal{H}$, and is analogous to the standard operator norm $\|S\|=\displaystyle \sup  \{ \|Sx \|: x\in \mathcal{H}, \|x\|=1 \}.$ In fact, for
every $S \in \mathcal{L(H)}$, the operator norm and the numerical radius is always comparable by the following inequality
$\frac{1}{2}\|S\|\leq w(S)\leq \|S\|.$
The numerical range of a bounded linear operator is convex, which is one of the most fundamental property of numerical range, which is called the \textit{Toeplitz-Hausdorff} Theorem. Other important properties are  the closure of $W(S)$ contains the spectrum of the operator $S$.
A Hilbert space $\mathcal{H}=\mathcal{H}(\Omega)$ of complex-valued functions on a nonempty open set $\Omega\subset \mathbb{C}$ which has the property that point evaluations are continuous, is called a  {\it functional Hilbert space}. 
The reproducing kernel of a functional Hilbert space $\mathcal{H}$ with $\{e_n\}$ as an orthonormal basis is  $K_\lambda(z)=\displaystyle \sum_n \overline{e_n(\lambda)}e_n(z)$.
Let $\widehat{k}_\lambda=K_\lambda/\|K_\lambda\|$ be the nomalized reproducing kernel of $\mathcal{H}$, where  $\lambda \in \Omega$. The function $\widetilde{A}$ defined on $\Omega$ by $\widetilde{A}(\lambda)=\langle A\widehat{k}_\lambda, \widehat{k}_\lambda\rangle$ is the {\it Berezin symbol} of a bounded linear operator $A$ on $\mathcal{H}$.
{\it Berezin range} and  {\it Berezin number} of the operator $A$ are defined by
$ \mbox{Ber}(A)=\{\widetilde{A}(\lambda):\lambda\in \Omega\}~ \mbox{and} ~\mbox{ber}(A)=\mbox{sup}\{|\widetilde{A}(\lambda)|:\lambda\in \Omega\},$  respectively.
Clearly, the  Berezin symbol $\widetilde{A}$ is a bounded function on $\Omega$ whose values lie in the numerical range of the operator $A$, and hence $ \mbox{Ber}(A)\subseteq W(A) ~\mbox{and}~ \mbox{ber}(A)\leq w(A).$

\subsection{Inclusion of zero}
Bourdon and Shapiro \cite{Bourdon_Shapiro_IEOT_2002} have proved that for any composition operator $C_\varphi$ not the identity map, $0\in \overline{W(C_\varphi)}$. They also showed that origin lies in the interior of $W(C_\varphi)$, whenever $C_\varphi$ fails to have dense range; if $\varphi=0$, then $0$ is in the interior of $W(C_\varphi)$ unless $\varphi(z)=\lambda z$ for some $\lambda\in [-1, 1];$ if $\varphi$ fixes a nonzero point in $\mathbb{D}$ and is neither the identity map nor a positive conformal dilation, then origin is in the interior of $W(C_\varphi)$. Same results hold for weighted composition operator, (see \cite{Gunatillake_Jovovic_Smith_JMAA_2014}). \\ \\
In this subsection, we prove that  $0\in \overline{W\left(\displaystyle \sum_{j=1}^{n}C_{\psi_j,\varphi_j}\right)}$, where $\varphi_j$ are analytic self-map of $\mathbb{D}$ which are not the identity maps and $\psi_j\in \mathcal{H}(\mathbb{D})~(j=1, 2,\dots, n)$. Lemma \ref{Lem_4.1} deals with the action of the adjoint of the composition operator $M_{\psi, \varphi}$ on Hardy space $\mathcal{H}^2(\mathbb{D})$.
\begin{lemma}\label{Lem_4.1}
	Let $\varphi_{j{(1\leq j\leq n)}}$ be analytic self-map of $\mathbb{D}$ and $\psi_j\in \mathcal{H}(\mathbb{D})$ such that $M_{\psi, \varphi}=\displaystyle\sum_{j=1}^{n}C_{\psi_j, \varphi_j} $ is bounded on $\mathcal{H}^2(\mathbb{D})$.	
	Then, $$ M_{\psi, \varphi}^*K_w=\sum_{j=1}^{n}\overline{\psi_j(w)}K_{\varphi_j(w)}~\textnormal{for}~w\in \mathbb{D}.$$
\end{lemma}

\begin{proof}
	Let $f\in \mathcal{H}^2(\mathbb{D})$, then we have
	\begin{align*}
	\langle f,  M_{\psi, \varphi}^*K_w \rangle_{\mathcal{H}^2(\mathbb{D})}&=  \langle  M_{\psi, \varphi}f, K_w \rangle_{\mathcal{H}^2(\mathbb{D})}\\
	&=\bigg\langle \sum_{j=1}^{n}C_{\psi_j, \varphi_j} , K_w \bigg\rangle_{\mathcal{H}^2(\mathbb{D})}\\
	&=\sum_{j=1}^{n}\bigg\langle \psi_j(fo \varphi_j) , K_w \bigg\rangle_{\mathcal{H}^2(\mathbb{D})}\\
	&=\sum_{j=1}^{n}\psi_j(w) f(\varphi_j(w)) \\
	&=\sum_{j=1}^{n}\bigg\langle  f, \overline{\psi_j(w)}K_{\varphi_j(w)} \bigg\rangle_{\mathcal{H}^2(\mathbb{D})}\\
	&=\bigg\langle  f, \sum_{j=1}^{n}\overline{\psi_j(w)}K_{\varphi_j(w)} \bigg\rangle_{\mathcal{H}^2(\mathbb{D})}.
	\end{align*}
	So, we have
	$$\bigg\langle f,  M_{\psi, \varphi}^*K_w \bigg\rangle_{\mathcal{H}^2(\mathbb{D})}=\bigg\langle  f,  \sum_{j=1}^{n}\overline{\psi_j(w)}K_{\varphi_j(w)} \bigg\rangle_{\mathcal{H}^2(\mathbb{D})}$$
	$\mbox{for~ all}~ f\in \mathcal{H}^2(\mathbb{D})$,	which leads the following
	$$ M_{\psi, \varphi}^*K_w=\sum_{j=1}^{n}\overline{\psi_j(w)}K_{\varphi_j(w)}.$$
\end{proof}
In the following theorem, we prove that the origin lies in the closure of the $W(M_{\psi, \varphi})$, where $M_{\psi, \varphi}=\displaystyle\sum_{j=1}^{n}C_{\psi_j, \varphi_j} $, $\varphi_{j{(1\leq j\leq n)}}$ are analytic self-maps of $\mathbb{D}$ which are not identities and $\psi_j\in \mathcal{H}(\mathbb{D})$. 
\begin{theorem}
		Let $\varphi_{j{(1\leq j\leq n)}}$ be analytic self-maps of $\mathbb{D}$ which are not identities and $\psi_j\in \mathcal{H}(\mathbb{D})$.
		Then $0\in \overline{W(M_{\psi, \varphi})}$.
		\end{theorem}

\begin{proof}
Let $\widehat{k}_z=\frac{K_z}{\|K_z\|}$ where $z\in \mathbb{D}$ and $K_z$ is the reproducing kernel at $z$. Let $$\lambda_z=\langle M_{\psi, \varphi}\widehat{k}_z, \widehat{k}_z \rangle=\frac{\langle M_{\psi, \varphi} K_z,~ K_z\rangle }{\|K_z\|^2}$$ be a point of $W(M_{\psi, \varphi})$. For each point $z\in \mathbb{D}$, we have
\begin{align}\label{Eq_Numerical_range}
\lambda_z&=\frac{\langle M_{\psi, \varphi} K_z,~ K_z\rangle }{\|K_z\|^2}\nonumber\\
&=\frac{\langle  K_z,~ M_{\psi, \varphi}^*K_z\rangle }{\|K_z\|^2}\nonumber\\
&=(1-|z|^2)\langle K_z, \sum_{j=1}^{n}\overline{\psi_j(z)}K_{\varphi_j(z)}\rangle\nonumber\\
&=(1-|z|^2)\left(\sum_{j=1}^{n}\overline{\psi_j(z)}K_z(\varphi_j(z))\right)\nonumber\\
&=\sum_{j=1}^{n}\overline{\psi_j(z)}\frac{(1-|z|^2)}{(1-\overline{z}\varphi_j(z))},
\end{align}
for $z\in \mathbb{D}$. 
 Also, $M_{\psi, \varphi}(1)=\psi\in \mathcal{H}^2$, and $\lim\limits_{r\rightarrow 1^-}\psi(r\mu)$ exists for almost all $\mu \in \mathbb{D}$ and 
 since $\varphi_{j{(1\leq j\leq n)}}$ are not identity maps, its boundary function differs from the identity almost everywhere. Therefore there exists $\xi \in \partial\mathbb{D}$ so that $\lim\limits_{r\rightarrow 1^-}\psi(r\xi)$ exists finitely and $\lim\limits_{r\rightarrow 1^-}\varphi_j(r\xi)\neq \xi$. Hence, from \eqref{Eq_Numerical_range}, we have $\lambda_{r\xi}\rightarrow 0$ as $r\rightarrow 1^-$, which implies that $0\in \overline{W(M_{\psi, \varphi})}$. 
\end{proof}

\subsection{Convexity of the Berezin range of composition operator}
The Hardy space \index{Hardy space} $\mathcal{H}^2(\mathbb{D})$ consisting of the functions $f(z)=\sum_{n=0}^\infty \widehat{f}(n)z^n$ holomorphic in the open unit disk $\mathbb{D}$ such that $\|f\|^2=\sum_{n=0}^\infty|\widehat{f}(n)|^2<\infty$, where $\widehat{f}(n)$ denotes the $n$-th Taylor coefficient of $f$. Let $\varphi$ be a holomorphic self-map of $\mathbb{D}$, then the equation $C_\varphi(f)=f\circ \varphi$ defines a composition operator $C_\varphi$ with inducing map $\varphi$. 
These operators are appreciated by many and have a significant and deep roots in operator and function theory (see \cite{Cowen_MacCluer_1995, Shapiro_1993}). Berger and Coburn \cite{Berger_Coburn_1986} asked: {\it if the Berezin symbol of an operator on the Hardy or Bergman space vanishes on the boundary of the disk, must the operator be compact? }
Nordgren and Rosenthal \cite{Nordgren_Rosenthal_1992} have demonstrated, on a so-called standard reproducing kernel Hilbert space, that if the Berezin symbols of all unitary equivalents of an operator vanish on the boundary, then the operator is compact. The counterexamples presented come in the form of composition operators.
Another significant finding from Axler and Zheng \cite{Axler_Zheng_1998}, is that if $S$ is a finite sum of finite products of Toeplitz operators acting on the Bergman space of the unit disk, then $S$ is compact iff the Berezin symbol of $S$ vanishes as it approaches the boundary of the disk. 
One motivation for studying the Berezin range of these operators is that they often elude  Axler-Zheng type results; e.g. there are composition operators such that $\widetilde{C}_\varphi(z) \rightarrow 0$ as $ z \rightarrow \partial \mathbb{D}$, but $C_\varphi$ is not compact.
From \textit{Toeplitz-Hausdorff} theorem, the natural question arises in our mind: \\ \\
{\bf Question-$3$:} {\bf Given a bounded operator acting on reproducing kernel Hilbert space, is it true that Berezin range of operator convex?}\\

By \textit{Toeplitz-Hausdorff} Theorem, the numerical range of an operator is always convex \cite{Gustafson_PAMS_1970}. It is not very difficult to observe that the Berezin range of an operator always a subset of the numerical range. In general the Berezin range of an operator need not be convex. Karaev \cite{Karaev} have initiated the study of the geometry of the Berezin range. In fact Karaev \cite{Karaev} have proved that the Berezin range of the Model operator $M_{z^n}$ on the model space is convex.  Motivated by this, recently Cowen and Felder \cite{CowenFelder} explored the convexity of the Berezin range for matrices, multiplication operators on reproducing kernel Hilbert space, composition operators acting on Hardy space of the unit disc. In light of the Cowen-Felder's results, Augustine {\it et al.} \cite{Augustine_Garayev_Shankar_CAOT_2023} have characterized the convexity of the Berezin range of composition operators acting on Bergman space over the unit disk.  
The Berezin transform is most useful on any reproducing kernel Hilbert space (RKHS), and the Berezin range may naturally be more relevant in this context. Now, we propose the following problem.
\\ \\
{\bf Question-$4$:  What can be said about the convexity of the Berezin range of special class of operators acting on $\mathcal{H}_\gamma(\mathbb{D})$?}
\\
We start by investigating the convexity of the Berezin range of composition operator on $\mathcal{H}_\gamma(\mathbb{D})$.

Among the composition operator results presented in \cite{Augustine_Garayev_Shankar_CAOT_2023, CowenFelder}, the main obstacle to obtaining analogous results in the reproducing kernel Hilbert space $\mathcal{H}_\gamma(\mathbb{D})$ is the greater complexity of the reproducing kernel, which is given by $ K_w(z)=(1-\overline{w}z)^{-\gamma}.$
Here, we obtain the values of $\beta$ for which the Berezin range of $C_\varphi$ is convex.
\subsubsection{\bf Elliptic symbols}
Theorem \ref{Thm_Convexity_1} describes the convexity of the Berezin range of composition operators in terms of elliptic automorphisms. 
\begin{theorem}\label{Thm_Convexity_1}
	Let $\gamma \in \mathbb{N},~ \beta\in \overline{\mathbb{D}}$ and $\varphi(w)=\beta w$. Then Berezin range of $C_\varphi$ acting on $\mathcal{H}_\gamma(\mathbb{D})$ is convex if, and only if, $-1\leq \beta\leq 1$.
\end{theorem} 
\begin{proof}
Let	$\beta\in \overline{\mathbb{D}}$ and $w\in \mathbb{D}$. Put $w=re^{i\theta}$, $0\leq r<1$, then we have
	\begin{align*}
\widetilde{C_\varphi}(w)&=\langle C_\varphi \widehat{k}_w, \widehat{k}_w \rangle\\
&=\frac{(1-|w|^2)^\gamma}{(1-|w|^2\beta)^\gamma}\\
&=\frac{(1-r^2)^\gamma}{(1-r^2\beta)^\gamma}.	
\end{align*}
If $\beta=1$, then $\widetilde{C_\varphi}(re^{i\theta})=\frac{(1-r^2)^\gamma}{(1-r^2\beta)^\gamma}=1$.
So, Ber$(C_\varphi)=\{1\}$, which is convex. Similarly for $\beta\in [-1, 1]$, we have
$\widetilde{C_\varphi}(re^{i\theta})=\frac{(1-r^2)^\gamma}{(1-r^2\beta)^\gamma}$ 
Therefore, $$Ber(C_\varphi)=\left\{\frac{(1-r^2)^\gamma}{(1-r^2\beta)^\gamma}: r\in [0, 1)\right\}=(0, 1],$$ which is also convex.
Conversly, suppose that Ber$(C_\varphi)$ is convex. So, we have  $\widetilde{C_\varphi}(re^{i\theta})=\frac{(1-r^2)^\gamma}{(1-r^2\beta)^\gamma}$, which is a function independent of $\theta$. So, if Ber$(C_\varphi)$ is convex, it must be either a point or a line segment. Therefore it gives that Ber$(C_\varphi)$ is a point if, and only if, $\beta=1$, so let us assume that Ber$(C_\varphi)$ is a line segment. Note that $\widetilde{C_\varphi}(0)=1$ and $\lim\limits_{r\rightarrow 1^-}\widetilde{C_\varphi}(re^{i\theta})=0$. This shows that Ber$(C_\varphi)$ must be a line segment passing through the point $1$ and approaching to the origin. Consequently, we must have the imaginary part of $\beta$ is zero, which can happen if, and only if, the imaginary part of $\beta$ is zero. Since $\beta\in \overline{\mathbb{D}}$, we have $-1\leq \beta\leq 1$. Therefore, if Ber$(C_\varphi)$ is convex then $-1\leq \beta \leq 1$. 
\end{proof}

\begin{remark}
	\begin{enumerate}
	\item[(i)]		If we choose $\gamma=2$, $\beta\in \overline{\mathbb{D}}$ and $\varphi(w)=\beta w$, then the Berezin range of $C_\varphi$ acting on $\mathcal{L}_a^2(\mathbb{D})$ is convex if, and only if, $-1\leq \beta\leq 1$, which has been proved in \cite{Augustine_Garayev_Shankar_CAOT_2023}.
		\item[(ii)] If we choose $\gamma=2$, $\beta\in \overline{\mathbb{D}}$ and $\varphi(w)=\beta w$. Then the Berezin range of $C_\varphi$ acting on $\mathcal{H}^2(\mathbb{D})$ is convex if, and only if, $-1\leq \beta\leq 1$, which has been proved in \cite{Augustine_Garayev_Shankar_CAOT_2023}.
		\end{enumerate}

\end{remark}

\begin{corollary}\label{Cor_4.5}
	Let $\beta\in \mathbb{T}$ and $\varphi(w)=\beta w$. Then Berezin range of $C_\varphi$ acting on $\mathcal{H}_\gamma(\mathbb{D})$ is convex if and only if $ \beta= 1$ or $ \beta=-1$.
\end{corollary}

As a consequence of this Corollary \ref{Cor_4.5}, for Hardy and Bergman case, one can see the works of \cite[Theorem 4.1]{CowenFelder} and \cite[Corollary 5.2]{Augustine_Garayev_Shankar_CAOT_2023}.

Next, we move to investigate the convexity of the Berezin range of composition operator with another symbol.

\subsubsection{\bf Blaschke factor symbol}
Let the automorphism of the unit disk (known as the {\it Blaschke factor}) given by
\begin{equation}\label{eq:blaschke_factor}
    \varphi_\alpha (z)=\dfrac{z-\alpha}{1-\overline{\alpha}z}
\end{equation}
where $\alpha \in \mathbb{D}$. 
The composition operator $C_{\varphi_\alpha}$  acting on $\mathcal{H}_\gamma(\mathbb{D}),$ where $\gamma\in \mathbb{N}$ is $C_{\varphi_\alpha}f=f \circ\varphi_\alpha$.
The Berezin transform of the composition operator $C_{\varphi_\alpha}$ at a point $w\in \mathbb{D}$ is given by

\begin{equation}\label{eq:berezin_transform}
    \begin{aligned}[b]
        \widetilde{C}_{\varphi_\alpha}(w)&=\langle C_{\varphi_\alpha}\widehat{k}_w,\widehat{k}_w\rangle \\
        &=\dfrac{1}{\|k_w\|^2}\langle C_{\varphi_\alpha}k_w,k_w\rangle \\
        &=\left(1-|w|^2\right)^\g \left(C_{\varphi_\alpha}k_w\right)(w)\\
        &=\left(1-|w|^2\right)^\gamma k_w\left(\varphi_\alpha(w)\right)\\
        &=\frac{\left(1-|w|^2\right)^\gamma}{\left(1-\overline{w}\varphi_\alpha(w)\right)^\gamma}\\
        &=\left(\frac{1-|w|^2}{1-\overline{w}\left(\frac{w-\alpha}{1-\overline{\alpha}w}\right)}\right)^\gamma\\
        &=\left(\frac{\left(1-|w|^2\right)\left(1-\overline{\alpha}w\right)}{1-|w|^2+\alpha\overline{w}-\overline{\alpha}w}\right)^\gamma.
    \end{aligned}
\end{equation}
The Berezin range of $C_{\varphi_\alpha}$, Ber$(C_\varphi)=\left\{\widetilde{C}_{\varphi_\alpha}(w): w\in \mathbb{D}\right\}$ is geometrically more interesting than the elliptic case, not always convex(!).
Finally, we aim to characterize the convexity of the Berezin range in this case. In the following, we try to find the values of $\alpha$ for which Ber$(C_{\varphi_\alpha})$ is convex.\\

We can utilize this information to derive key insights into the geometry of Ber$(C_{\varphi_\alpha})$.
\begin{theorem}\label{th:symmetric}
    Let for $\alpha\in \mathbb{D},\;\varphi_\alpha$ be the Blaschke factor given by \eqref{eq:blaschke_factor}. Then the Berezin range of the composition operator $C_{\varphi_\alpha}$ acting on $\mathcal{H}_\gamma(\mathbb{D})$ is closed under complex conjugation and hence is symmetric about the real axis.
\end{theorem}
\begin{proof}
    Let $|\alpha|=\rho,\;|w|=r,$ where $w\in \mathbb{D}$. Assume that $w\neq 0,$ set $\lambda=\frac{r^2}{\rho^2}.\frac{\alpha^2}{w}$. We claim that $\overline{\widetilde{C}_{\varphi_\alpha}(w)}=\widetilde{C}_{\varphi_\alpha}(\lambda)$. By \eqref{eq:berezin_transform}, we have
    \begin{equation*}
        \widetilde{C}_{\varphi_\alpha}(\lambda)=\left(\frac{\left(1-|\lambda|^2\right)\left(1-\overline{\alpha}\lambda\right)}{1-|\lambda|^2+\alpha\overline{\lambda}-\overline{\alpha}\lambda}\right)^\gamma.
    \end{equation*}
Notice that $|\lambda|=\frac{r^2}{\rho^2}\times\frac{|\alpha|^2}{|w|}=r$, $\overline{\alpha}\lambda=\frac{|\alpha|^2}{\alpha}\times\frac{r^2}{\rho^2}\times\frac{\alpha^2}{w}=r^2\times\frac{\alpha}{w}$, and $\alpha\overline{\lambda}-\overline{\alpha}\lambda=-2i \Im\left\{\overline{\alpha}\lambda\right\}=-2i \Im\left\{\frac{r^2\alpha}{w}\right\}$.
   Therefore,
    \begin{equation*}
        \widetilde{C}_{\varphi_\alpha}(\lambda)=\left(\frac{(1-r^2)\left(1-\frac{r^2\alpha}{w}\right)}{1-r^2-2i \Im\left\{\frac{r^2\alpha}{w}\right\}}\right)^\gamma,
    \end{equation*}
and hence
    \begin{equation}\label{Eq_conju}      \overline{\widetilde{C}_{\varphi_\alpha}(\lambda)}=\left(\frac{(1-r^2)\left(1-\frac{r^2\overline{\alpha}}{\overline{w}}\right)}{1-r^2+2i \Im\left\{\frac{r^2\alpha}{w}\right\}}\right)^\g.
    \end{equation}
    Since, $1-\frac{r^2\overline{\alpha}}{\overline{w}}=1-\frac{|w|^2\overline{\alpha}}{\overline{w}}=1-w\overline{\alpha}$ and $2i \Im\left\{\frac{r^2\alpha}{w}\right\}=2i \Im\left\{\frac{|w|^2\alpha}{w}\right\}=2i \Im\left\{\alpha\overline{w}\right\}=\alpha\overline{w}-\overline{\alpha}w,$ and \eqref{Eq_conju} reduces to
    \begin{equation*}
        \overline{\widetilde{C}_{\varphi_\alpha}(\lambda)}=\left(\frac{(1-r^2)\left(1-\overline{\alpha}w\right)}{1-r^2+\alpha\overline{w}-\overline{\alpha}w}\right)^\gamma=\widetilde{C}_{\varphi_\alpha}(w),
    \end{equation*}
    which gives that 
    \begin{equation*}
        \overline{\widetilde{C}_{\varphi_\alpha}(w)}=\widetilde{C}_{\varphi_\alpha}(\lambda).
    \end{equation*}
    Thus, our claim is proved. As $|\lambda|=\frac{r^2}{\rho^2}\times\frac{|\alpha|^2}{|w|}=|w|<1,\;\lambda\in \mathbb{D},$ and hence $\widetilde{C}_{\varphi_\alpha}(\lambda)\in \mbox{Ber}(C_{\varphi_\alpha})$, i.e., $\overline{\widetilde{C}_{\varphi_\alpha}(w)}$. If, however, $w=0,$ then $\widetilde{C}_{\varphi_\alpha}(w)=\widetilde{C}_{\varphi_\a}(0)=1$, a real number, implying that $\overline{\widetilde{C}_{\varphi_\alpha}(0)}=\widetilde{C}_{\varphi_\alpha}(0)$.
    This completes the proof. 
\end{proof}
We highlight a corollary of this result that plays a key role in establishing the characterization of convexity.
\begin{corollary}\label{cor:real_in_Berezin_range}
        If the Berezin range of $C_{\varphi_\alpha}$ acting on $\mathcal{H}_\gamma(\mathbb{D})$ be convex, then $\Re\left\{\widetilde{C}_{\varphi_\alpha}(z)\right\}\in \textnormal{Ber}(C_{\varphi_\alpha})$ for every $z\in \mathbb{D}$.
    \end{corollary}
    \begin{proof}
    Since $z\in \mathbb{D},\;\widetilde{C}_{\varphi_\alpha}(z)\in \textnormal{Ber}(C_{\varphi_\alpha})$. Then by Theorem \ref{th:symmetric}, $\overline{\widetilde{C}_{\varphi_\alpha}(z)}\in \mbox{Ber}(C_{\varphi_\alpha})$. Suppose $\mbox{Ber}(C_{\varphi_\alpha})$ be convex, then we must have
    \begin{equation*}
        \frac{1}{2}\widetilde{C}_{\varphi_\alpha}(z)+\frac{1}{2}\overline{\widetilde{C}_{\varphi_\alpha}(z)}=\Re\left\{\widetilde{C}_{\varphi_\alpha}(z)\right\}\in\mbox{Ber}(C_{\varphi_\alpha}).
    \end{equation*}
     \end{proof}
    
Now, we are in a position to provide a characterization of the convexity of the composition operator acting on $\mathcal{H}_\gamma(\mathbb{D})$. Here, we highlight that we have identified a mistake in a recently published article, specifically in \cite[Lemma 5.4 and Theorem 5.7]{Augustine_Garayev_Shankar_CAOT_2023}. Moreover, the factor $$K_{\alpha,z}=\dfrac{1-|z|^2}{(1-|z|^2+2i\Im\left\{\alpha\overline{z}\right\})}$$ should be $$K_{\alpha,z}=\dfrac{1-|z|^2}{\left(1-|z|^2\right)^2+4\left(\Im\left\{\alpha\overline{z}\right\}\right)^2}$$ in \cite[Lemma 5.4]{Augustine_Garayev_Shankar_CAOT_2023} and $\widetilde{C}_{\varphi_\alpha}(w)=\left(1-r|\alpha|^2\right)^4$ should be $\widetilde{C}_{\varphi_\alpha}(w)=\left(1-r|\alpha|^2\right)^2$ in \cite[Theorem 5.7]{Augustine_Garayev_Shankar_CAOT_2023}.
The authors of \cite{CowenFelder, Augustine_Garayev_Shankar_CAOT_2023} characterized explicitly both real and imaginary parts of the Berezin transform of composition operators on Hardy and Bergman space.  
The main obstacle to obtaining both real and imaginary parts of the Berezin transform of composition operators in reproducing kernel Hilbert space $\mathcal{H}_\gamma(\mathbb{D})$ is the increased complexity of the reproducing kernel, $ K_w(z)=(1-\overline{w}z)^{-\gamma}.$ We have adopted a slightly different method to characterize the convexity of the composition operator in reproducing kernel Hilbert space $\mathcal{H}_\gamma(\mathbb{D})$.

\begin{theorem}\label{th:value_of_alpha}
         The Berezin range of $C_{\varphi_\alpha}$ acting on $\mathcal{H}_\gamma(\mathbb{D})$, $\gamma\in \mathbb{N}$ is convex if, and only if, $\alpha=0$.
     \end{theorem}

\begin{proof}
    When $\alpha=0,\;\varphi_\alpha(z)=z.$ In this case, for $w\in \mathbb{D}$, it follows from \eqref{eq:berezin_transform} that 
         \begin{equation*}
             \widetilde{C}_{\varphi_\alpha}(w)=\left(\frac{1-|w|^2}{1-|w|^2}\right)^\gamma=1.
         \end{equation*}
         So, $\mbox{Ber}(C_{\varphi_\alpha})=\{1\}$, which being a singleton set is convex. Conversely, assume that $\mbox{Ber}(C_{\varphi_\alpha})$ be convex. Then from Corollary \ref{cor:real_in_Berezin_range}, $\Re\left\{\widetilde{C}_{\varphi_\alpha}(z)\right\} \in \mbox{Ber} (C_{\varphi_\alpha})$. This means that for every $z\in \mathbb{D}$, there exists $w\in \mathbb{D}$ such that
         \begin{equation*}
             \widetilde{C}_{\varphi_\alpha}(w)=\Re\left\{\widetilde{C}_{\varphi_\alpha}(z)\right\},
         \end{equation*}
         and hence

\begin{equation}\label{eq:imaginary_vanishes}
             \Im\left\{\widetilde{C}_{\varphi_\alpha}(w)\right\}=0.
         \end{equation}

Now,
         \begin{equation*}
            \begin{split}
                \widetilde{C}_{\varphi_\alpha}(w)&=\left(\frac{\left(1-|w|^2\right)\left(1-\overline{\alpha}w\right)}{1-|w|^2+\alpha\overline{w}-\overline{\alpha}w}\right)^\gamma\\
                &=\left(\frac{\left(1-|w|^2\right)\left(1-\overline{\alpha}w\right)}{1-|w|^2+2i \Im\left\{\alpha\overline{w}\right\}}\right)^\gamma \\
                &=\left(\frac{\left(1-|w|^2\right)\left(1-\overline{\alpha}w\right)\left(1-|w|^2-2i \Im\left\{\alpha\overline{w}\right\}\right)}{\left(1-|w|^2\right)^2+4\left(\Im\left\{\alpha\overline{w}\right\}\right)^2}\right)^\gamma \\
                &=K_{\alpha,w}^\gamma\left(1-|w|^2-2i\Im\left\{\alpha\overline{w}\right\}-\overline{\alpha}w\left(1-|w|^2\right)+2i \Im\left\{\alpha\overline{w}\right\}\overline{a}w\right)^\gamma\\ 
                &=K_{\alpha,w}^\gamma\left(1-|w|^2+2i\Im\left\{\overline{\alpha}w\right\}-\overline{\alpha}w\left(1-|w|^2\right)+2i \Im\left\{\alpha\overline{w}\right\}\overline{a}w\right)^\gamma,
        \end{split}
         \end{equation*}
       where  $K_{\alpha,w}=\dfrac{1-|w|^2}{\left(1-|w|^2\right)^2+4\left(\Im\left\{\alpha\overline{w}\right\}\right)^2}$.
                  Replacing $\overline{\alpha}w$ by $\left(\Re\left\{\overline{\alpha}w\right\}+i \Im\left\{\overline{\alpha}w\right\}\right)$ above, becomes
          \begin{align*}
  \widetilde{C}_{\varphi_\alpha}(w)=K_{\alpha,w}^\gamma&\left(1-|w|^2+2i \Im\left\{\overline{\alpha}w\right\}-\left(\Re\left\{\overline{\alpha}w\right\}+i \Im\left\{\overline{\alpha}w\right\}\right)\left(1-|w|^2\right)\right.\\
  &\left. \quad-2i \Im\left\{\overline{\alpha}w\right\}\left(\Re\left\{\overline{\alpha}w\right\}+i \Im\left\{\overline{\alpha}w\right\}\right)\right)^\gamma\\
  =K_{\alpha,w}^\gamma &\left(1-|w|^2-\left(1-|w|^2\right)\Re\left\{\overline{\alpha}w\right\}+2\left(\Im\left\{\overline{\alpha}w\right\}\right)^2\right.\\
  &\left. \quad+i\left(2\Im\left\{\overline{\alpha}w\right\}-\left(1-|w|^2\right)\Im\left\{\overline{\alpha}w\right\}-2\Im\left\{\overline{\alpha}w\right\}\Re\left\{\overline{\alpha}w\right\}\right)\right)^\gamma\\
  =K_{\alpha,w}^\gamma&\left(\left(1-|w|^2\right)\left(1-\Re\left\{\overline{\alpha}w\right\}\right)+2\left(\Im\left\{\overline{\alpha}w\right\}\right)^2+ i \Im\left\{\overline{\alpha}w\right\}\left(1+|w|^2-2 \Re\left\{\overline{\alpha}w\right\}\right)\right)^\gamma.
           \end{align*}

        We put
        \begin{equation*}
        \begin{split}
            &R\cos{\Theta}=\left(1-|w|^2\right)\left(1-\Re\left\{\overline{\alpha}w\right\}\right)+2\left(\Im\left\{\overline{\alpha}w\right\}\right)^2\\
            &R\sin{\Theta}=\Im\left\{\overline{\alpha}w\right\}\left(1+|w|^2-2 \Re\left\{\overline{\alpha}w\right\}\right).
            \end{split}
        \end{equation*}
        Notice that $R=0$ when $\widetilde{C}_{\varphi_\alpha}(w)=0,$ i.e., when $K_{\alpha,w}=0$, i.e., when $|w|^2=1$; which cannot happen, since $w\in \mathbb{D}$. Switching to polar coordinates and recalling the definition of complex exponents $a^z=\exp\left({z\mbox{Log}a}\right)$, we have
        \begin{equation*}
            \begin{split}
                \widetilde{C}_{\varphi_\alpha}(w)&=K_{\alpha,w}^\gamma\left(R(\cos{\Theta}+i\sin{\Theta})\right)^\gamma\\
                &=K_{\alpha,w}^\gamma \exp{\left(\gamma\mbox{Log}\left(R(\cos{\Theta}+i\sin{\Theta})\right)\right)}\\
                &=K_{\alpha,w}^\gamma \exp{\left(\gamma(\mbox{log}R+i\Theta)\right)};\;\;\mbox{considering the principal value only}\\
                &=K_{\alpha,w}^\gamma e^{\gamma\mbox{Log}R}e^{i\Theta}\\
                &=K_{\alpha,w}^\gamma e^{\gamma\mbox{Log}R}(\cos{\Theta}+i\sin{\Theta}).
            \end{split}
        \end{equation*}
       By (\ref{eq:imaginary_vanishes}),
       \begin{equation}\label{Eq_0010}
          R\sin{\Theta}=0,\hspace{1 cm}\mbox{i.e.,}\hspace{1 cm}\Im\left\{\overline{\alpha}w\right\}\left(1+|w|^2-2 \Re\left\{\overline{\alpha}w\right\}\right)=0.
       \end{equation}
       Since for any $w,\alpha\in \mathbb{D},\;\left(1+|w|^2-2 \Re\left\{\overline{\alpha}w\right\}\right)>0,$ \eqref{Eq_0010} will be true iff $\Im\left\{\overline{\alpha}w\right\}=0.$ This says that $\alpha$ and $w$ lie on a straight line passing through the origin. So, we put $w=r\alpha$ for some $r\in \left(-\frac{1}{|\alpha|},\frac{1}{|\alpha|}\right),$ and then we find that
       \begin{equation*}
           \begin{split}
               \widetilde{C}_{\varphi_\alpha}(w)&=K_{\alpha,w}^\gamma\left((1-r^2|\alpha|^2)\left(1-\Re\left\{r|\alpha|^2\right\}\right)\right)^\gamma\\
               &=K_{\alpha,w}^\gamma\left((1-r^2|\alpha|^2)\left(1-r|\alpha|^2\right)\right)^\gamma\\
               &=\left(\dfrac{1-|w|^2}{\left(1-|w|^2\right)^2+4\left(\Im\left\{\alpha\overline{w}\right\}\right)^2}\right)^\gamma\left((1-r^2|\alpha|^2)\left(1-r|\alpha|^2\right)\right)^\gamma\\
               &=\left(\frac{1-r^2|\alpha|^2}{\left(1-r^2|\alpha|^2\right)^2}\right)^\gamma\left((1-r^2|\alpha|^2)\left(1-r|\alpha|^2\right)\right)^\gamma\\
               &=\left(1-r|\alpha|^2\right)^\gamma.
           \end{split}
       \end{equation*}
       Consequently, $\left\{\widetilde{C}_{\varphi_\alpha}(r\alpha):r\in \left(-\frac{1}{|\alpha|},\frac{1}{|\alpha|}\right)\right\}=\left((1-|\alpha|)^\gamma,(1+|\alpha|)^\gamma\right)$.\\ Now,
       \begin{equation*}
           \widetilde{C}_{\varphi_\alpha}(z)=\left(\dfrac{\left(1-|z|^2\right)\left(1-\overline{\alpha}z\right)}{1-|z|^2+\alpha\overline{z}-\overline{\alpha}z}\right)^\gamma,
       \end{equation*}
       and hence,
       \begin{equation*}
           \widetilde{C}_{\varphi_\alpha}(\rho e^{i\theta})=\left(\frac{\left(1-\rho^2\right)\left(1-\overline{\alpha}\rho e^{i\theta}\right)}{1-\rho^2+\alpha\rho e^{-i\theta}-\overline{\alpha}\rho e^{i\theta}}\right)^\gamma.
       \end{equation*}
       Letting $\rho \rightarrow -1,$ we have
       \begin{equation*}
          \lim_{\rho \rightarrow -1}{\widetilde{C}_{\varphi_\alpha}(\rho e^{i\theta})}=\begin{cases}
              1,\;\mbox{when}\;\alpha=0\\
              0,\;\mbox{when}\;\alpha\neq 0
          \end{cases}.
       \end{equation*}
     This shows that when $\alpha\neq0,$ given $\epsilon$ with $0<\epsilon<(1-|\alpha|)^\gamma$, there exists a point $z$ such that $\Re\left\{\widetilde{C}_{\varphi_\alpha}(z)\right\}<\epsilon.$ But if $\widetilde{C}_{\p_a}(w)=\Re\left\{\widetilde{C}_{\varphi_\alpha}(z)\right\},$ this is a contradiction because $$\widetilde{C}_{\p_a}(w)\in \left\{\widetilde{C}_{\varphi_\alpha}(r\alpha):r\in \left(-\frac{1}{|\alpha|},\frac{1}{|\alpha|}\right)\right\}=\left((1-|\alpha|)^\gamma,(1+|\alpha|)^\gamma\right).$$ Thus, $\mbox{Ber}(C_{\varphi_\alpha})$ cannot be convex unless $\alpha=0$.      
\end{proof}
It should be remarked here that the restriction of $\gamma \in \mathbb{N}$ is essential, as in general, for any $a, b \in \mathbb{C}$ and $\gamma>0$, $(ab)^\gamma\neq a^\gamma b^\gamma$, while the equality holds when $\gamma \in \mathbb{N}$, is used in the proof.\\

\subsubsection{\bf Geometric view of the Berezin range of composition operator.}

    In the following, with the aid of the computer, we plot some Berezin ranges of some composition operators induced by the Blaschke factor given by (\ref{eq:blaschke_factor}).       
       In Figure \ref{fig:comparison_a=0.5_g=1,2,3}, we see that the Berezin range of $C_{\phi_\alpha},\,Ber\left(C_{\phi_\alpha}\right)$ looks like a circular disc centered on $(1,0)$ with a hole in it in Hardy space while it looks like a cardioid in Bergman space. Also, for a fixed value of $\alpha$, the roughness of the outer boundary of the Berezin range of $\left(C_{\phi_\alpha}\right)$ increases with increasing $\gamma.$\\
           \par
        In Hardy space $(\g=1)$, if we vary $\alpha$ along the positive real axis, it can be seen that the size of the circular disc gradually gets bigger but the hole in it gets smaller when $\alpha$ approaches $1,$ and when $\alpha$ approaches the origin,  the disc gradually squeezes but the hole inside it expands. Finally, when $\alpha=0,$ this disc ultimately squeezes to its center. Similar happens if $\alpha$ varies along the negative real axis (see Figure \ref{fig:comparison_g=1_a=real}). Indeed, not only along the real axis, this happens if $\alpha$ varies along any straight line through the origin (see Figure \ref{fig:comparison_g=1_line_y=x}). Also, as $\alpha$ varies away from the origin, the smoothness of the boundary of the Berezin range gradually vanishes. \\  
    \par
    
    Likewise in the Bergman space $(\g=2).$ To be explicit, if we let $\alpha$ to vary along any straight line through the origin, away from the origin to the boundary of the unit disc $\D,$ the cardioid begins to expand with more rough boundary whereas the hole in it becomes smaller, and very tiny when $\alpha$ is near the boundary of $\D$ (see Figure \ref{fig:comparison_g=2_line_y=2x}).\\   
    \par
    
    In Figure \ref{fig:comparison_g=1_a=real} we see that in Hardy space, $Ber\left(C_{\phi_\alpha}\right)=Ber\left(C_{\phi_{-\alpha}}\right),$ when $\alpha$ is on the real axis. Indeed, this is true in any space and for any $\alpha \in \D,$ which we prove in the following theorem.
    \begin{theorem}\label{positive=negative}
    For any value of $\g$ and for any $\alpha \in \D,$
    \begin{equation}        Ber\left(C_{\phi_\alpha}\right)=Ber\left(C_{\phi_{-\alpha}}\right),
    \end{equation}
    in the space $\mathcal{H}_\g(\D).$
    \end{theorem}
    \begin{proof}
        First, we prove that $Ber\left(C_{\phi_\alpha}\right)\subset Ber\left(C_{\phi_{-\alpha}}\right),$ {\it i.e.}, every member of $Ber\left(C_{\phi_\alpha}\right)$ is equal to some element in $Ber\left(C_{\phi_{-\alpha}}\right).$ Let $\omega \in \D$ be arbitrary. Our intention is to find some $\omega '\in \D$ such that $\widetilde{C}_{\phi_\alpha}(\omega)=\widetilde{C}_{\phi_{-\alpha}}(\omega').$ Now,
        \begin{equation*}
            \begin{split}
                \widetilde{C}_{\phi_\alpha}(\omega)&=\left[\dfrac{\left(1-|\o|^2\right)\left(1-\overline{\a}\o\right)}{1-|\o|^2+\a\overline{\o}-\overline{\a}\o}\right]^\g\;\;\;[\mbox{c.f.}\;(\ref{eq:berezin_transform})]\\
                &=\left[\dfrac{\left(1-|-\o|^2\right)\left(1-\overline{(-\a)}(-\o)\right)}{1-|-\o|^2+(-\a)\overline{(-\o)}-\overline{(-\a)}(-\o)}\right]^\g \\
                &=\widetilde{C}_{\phi_{-\alpha}}(-\omega).
            \end{split}
        \end{equation*}
        Since $-\omega \in \D,\;Ber\left(C_{\phi_\alpha}\right)\subset Ber\left(C_{\phi_{-\alpha}}\right).$ Likewise, $Ber\left(C_{\phi_{-\alpha}}\right)\subset Ber\left(C_{\phi_{\alpha}}\right).$ This completes the proof.
    \end{proof}




\begin{figure}[h!]
    \centering
        \begin{subfigure}[b]{0.3\textwidth}
        \includegraphics[width=\textwidth]{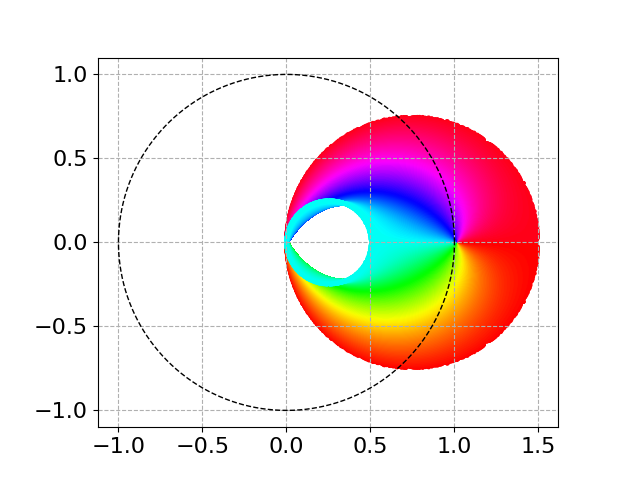}
        \caption{$\alpha=0.5,\;\g=1$}
        \label{fig:bf_g=1_a=0.5}
        \end{subfigure}
    \hfill
        \begin{subfigure}[b]{0.3\textwidth}
        \centering
        \includegraphics[width=\textwidth]{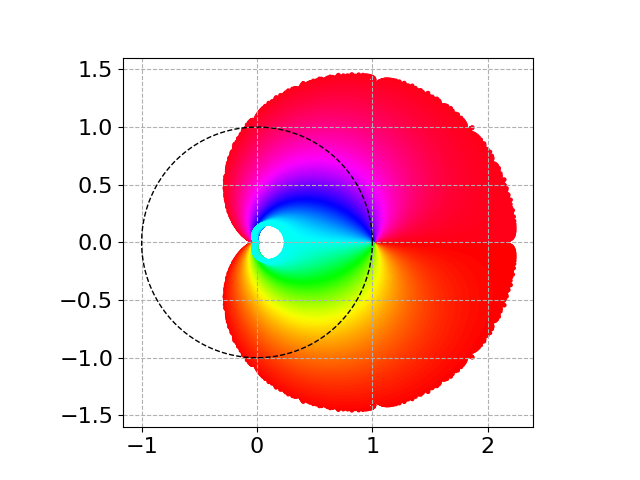}
        \caption{$\alpha=0.5,\;\g=2$}
        \label{fig:bf_g=2_a=0.5}
        \end{subfigure}
    \hfill
    \begin{subfigure}[b]{0.3\textwidth}
        \centering
        \includegraphics[width=\textwidth]{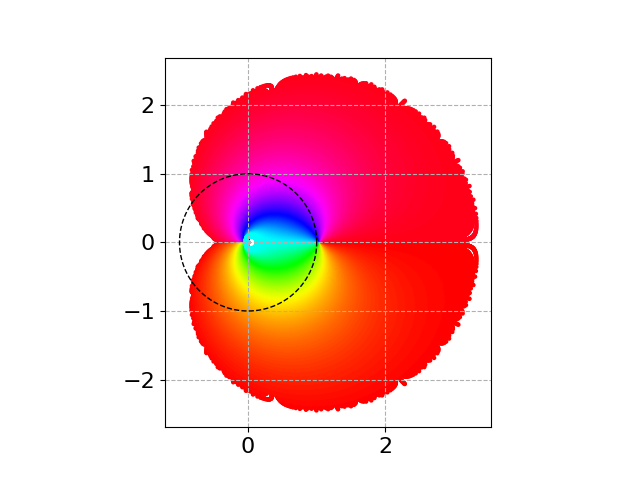}
        \caption{$\alpha=0.5,\;\g=3$}
        \label{fig:bf_g=3_a=0.5}
        \end{subfigure}
        \hfill
         \begin{subfigure}[b]{0.3\textwidth}
        \centering
        \includegraphics[width=\textwidth]{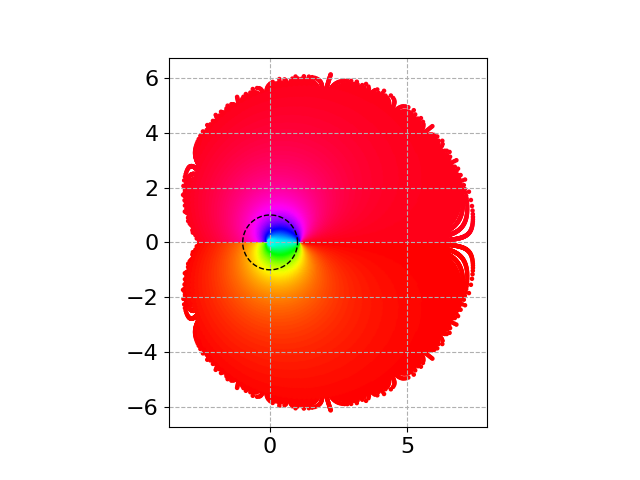}
        \caption{$\alpha=0.5,\;\g=5$}
        \label{fig:bf_g=5_a=0.5}
        \end{subfigure}
        \hfill
        \begin{subfigure}[b]{0.3\textwidth}
        \centering
        \includegraphics[width=\textwidth]{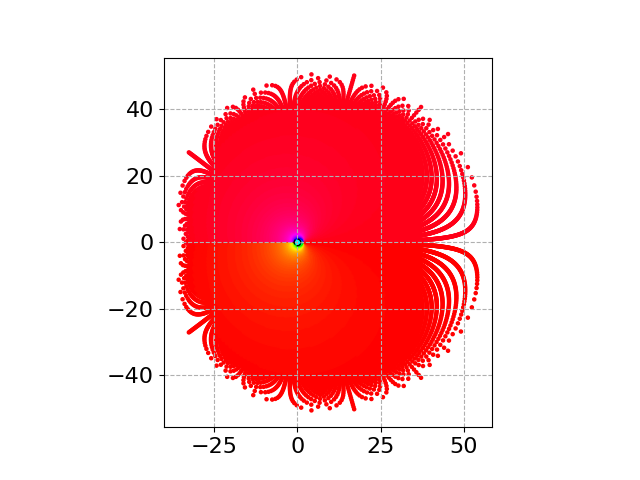}
        \caption{$\alpha=0.5,\;\g=10$}
        \label{fig:bf_g=10_a=0.5}
        \end{subfigure}
        \hfill
        \begin{subfigure}[b]{0.3\textwidth}
        \centering
        \includegraphics[width=\textwidth]{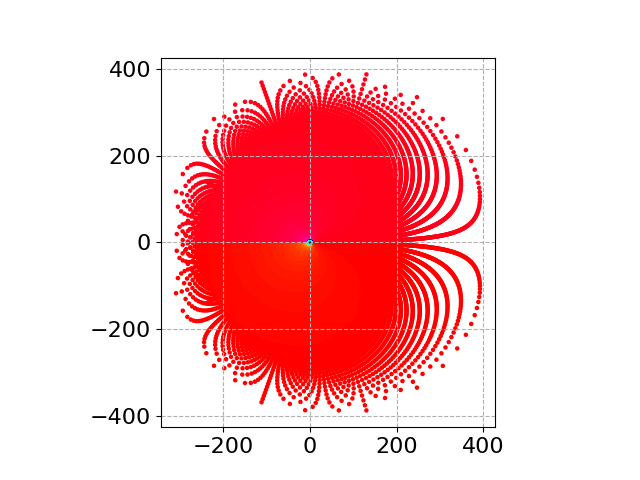}
        \caption{$\alpha=0.5,\;\g=15$}
        \label{fig:bf_g=15_a=0.5}
        \end{subfigure}
    \caption{Comparison of Berezin ranges of $C_{\phi_\alpha}$ for $\alpha=0.5$ in different spaces, viz. respectively in Hardy space, Bergman space and the other reproducing kernel Hilbert spaces.}
        \label{fig:comparison_a=0.5_g=1,2,3}
    \end{figure} 
    \begin{figure}[]
    \centering
        \begin{subfigure}{0.3\textwidth}
            \centering
            \includegraphics[width=\textwidth]{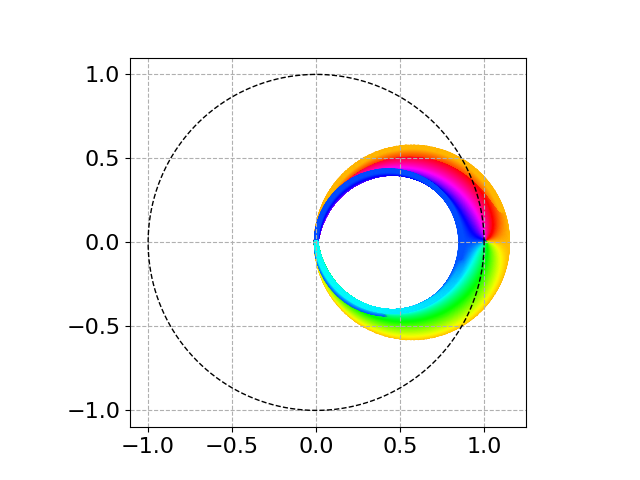}
            \caption{$\alpha=0.1+0.1i,\,\g=1$}
            \label{fig:bf_g=1_a=0.1+0.1i}
        \end{subfigure}
        \hfill
        \begin{subfigure}{0.3\textwidth}
            \centering
            \includegraphics[width=\textwidth]{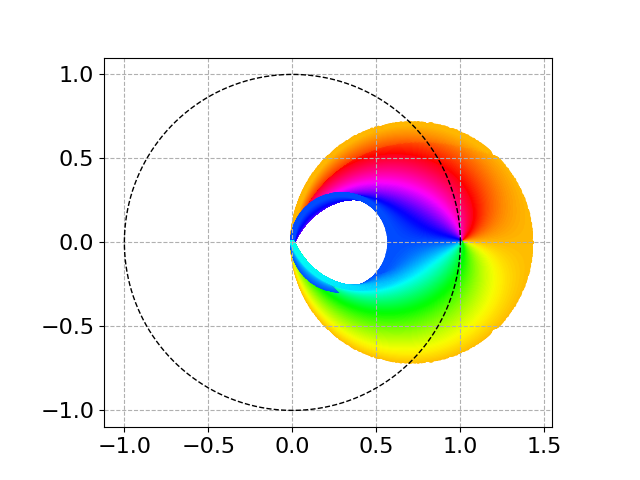}
            \caption{$\alpha=0.3+0.3i,\;\g=1$}
            \label{fig:bf_g=1_a=0.3+0.3i}
        \end{subfigure}
        \hfill
        \begin{subfigure}{0.3\textwidth}
            \centering
            \includegraphics[width=\textwidth]{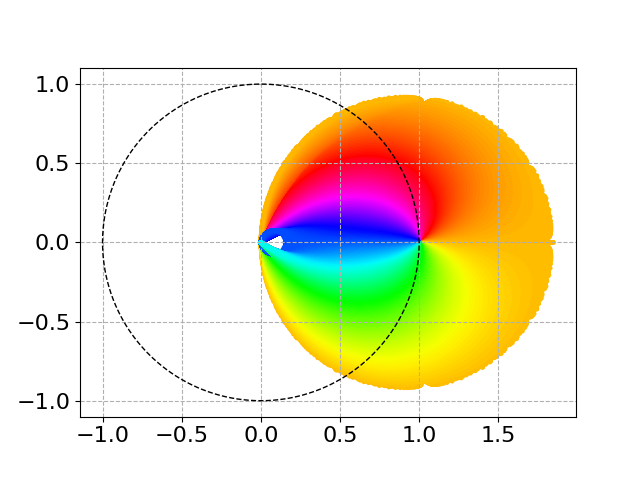}
            \caption{$\alpha=0.6+0.6i,\;\g=1$}
            \label{fig:bf_g=1_a=0.6+0.6i}
        \end{subfigure}      
    \caption{Comparison of Berezin ranges of $C_{\phi_\alpha}$ in Hardy space for different values of $\alpha$ when $\alpha$ varies along the line $y=x.$}
    \label{fig:comparison_g=1_line_y=x}
    \end{figure}

     \begin{figure}[]
    \centering
        \begin{subfigure}{0.3\textwidth}
              \centering
              \includegraphics[width=\textwidth]{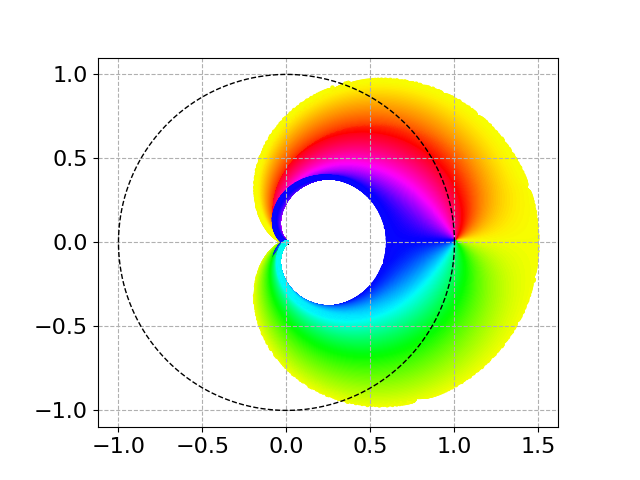}
              \caption{$\alpha=0.1+0.2i,\;\g=2$}
              \label{fig:bf_g=2_a=0.1+0.2i}
          \end{subfigure}  
          \hfill
          \begin{subfigure}{0.3\textwidth}
              \centering
              \includegraphics[width=\textwidth]{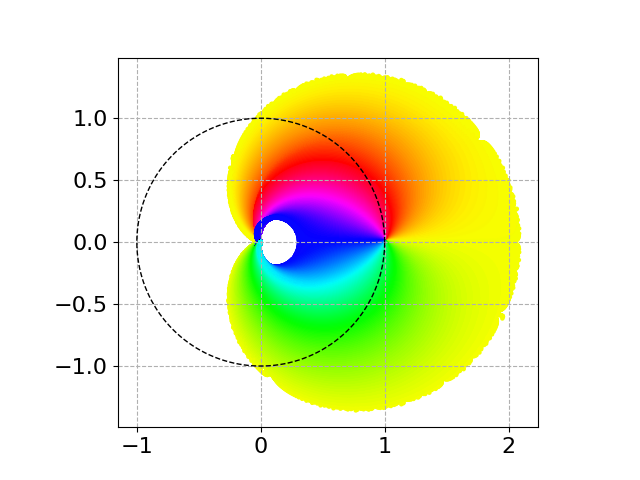}
              \caption{$\alpha=0.2+0.4i,\;\g=2$}
              \label{fig:bf_g=2_a=0.2+0.4i}
          \end{subfigure}
          \hfill
          \begin{subfigure}{0.3\textwidth}
              \centering
            \includegraphics[width=\textwidth]{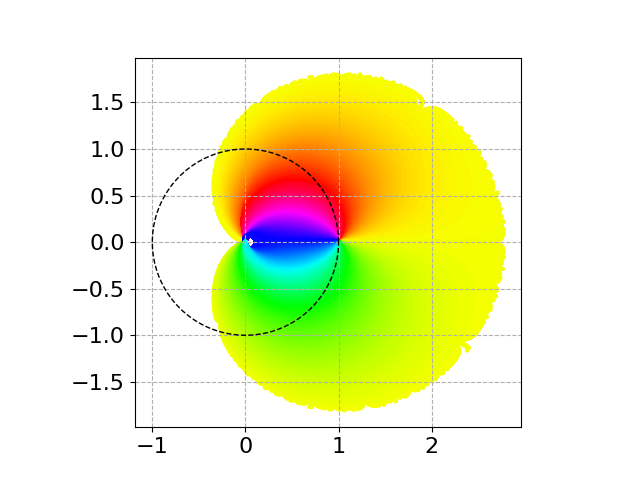}
            \caption{$\alpha=0.3+0.6i,\g=2$}
              \label{fig:bf_g=2_a=0.3+0.6i}
          \end{subfigure}
    \caption{Comparison of Berezin ranges of $C_{\phi_\alpha}$ in Bergman space for different values of $\alpha$ when $\alpha$ varies along the line $y=2x.$}
    \label{fig:comparison_g=2_line_y=2x}
    \end{figure}

\begin{figure}[]
\centering
\subfloat[$\alpha=-0.7,\;\gamma=1$]{%
\resizebox*{5cm}{!}{\includegraphics{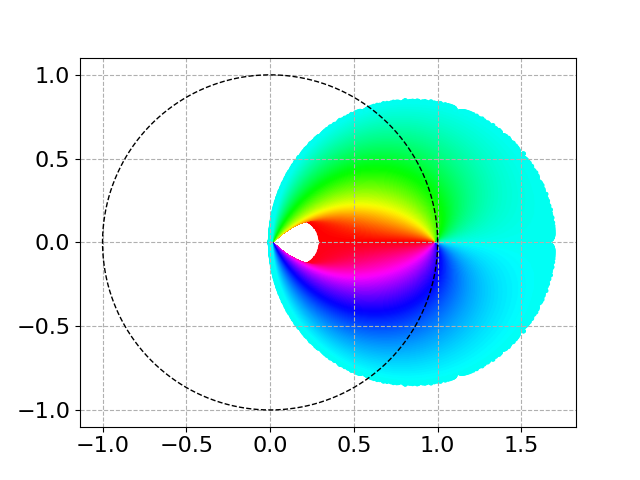}}}\hspace{5pt}
\subfloat[$\alpha=-0.3,\;\gamma=1$]{%
\resizebox*{5cm}{!}{\includegraphics{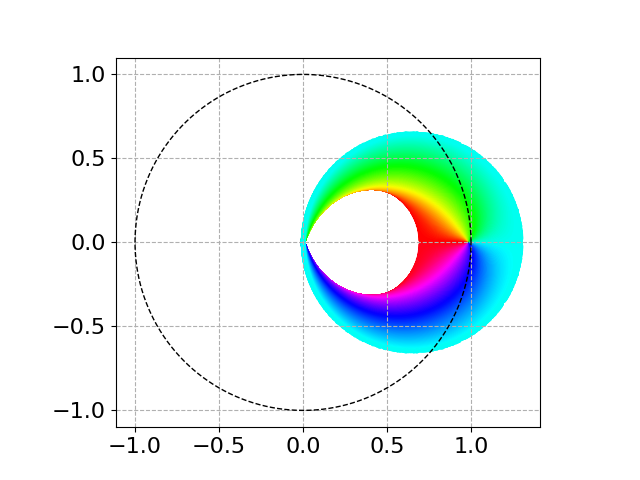}}}
\subfloat[$\alpha=-0.1,\;\gamma=1$]{%
\resizebox*{5cm}{!}{\includegraphics{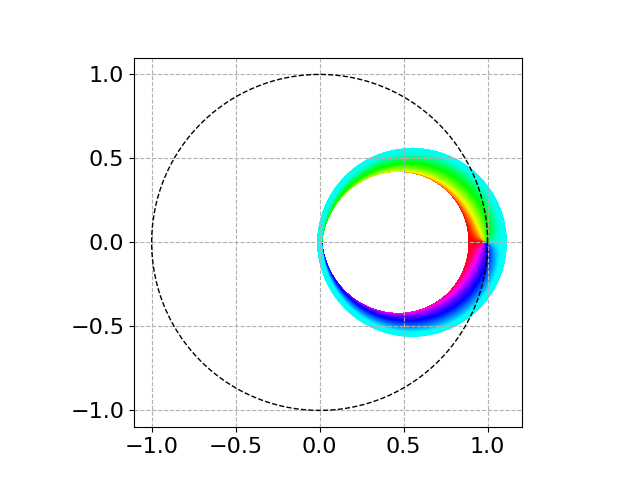}}}\hspace{5pt}

\subfloat[$\alpha=-10^{-3},\;\g=1$]{%
\resizebox*{5cm}{!}{\includegraphics{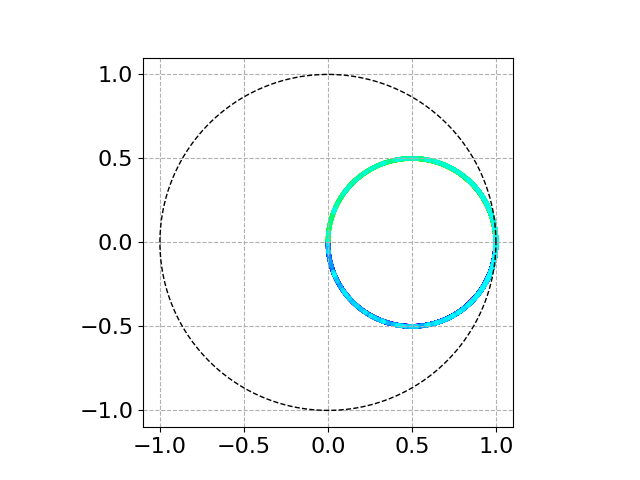}}}\hspace{5pt}
\subfloat[$\alpha=-10^{-5},\;\g=1$]{%
\resizebox*{5cm}{!}{\includegraphics{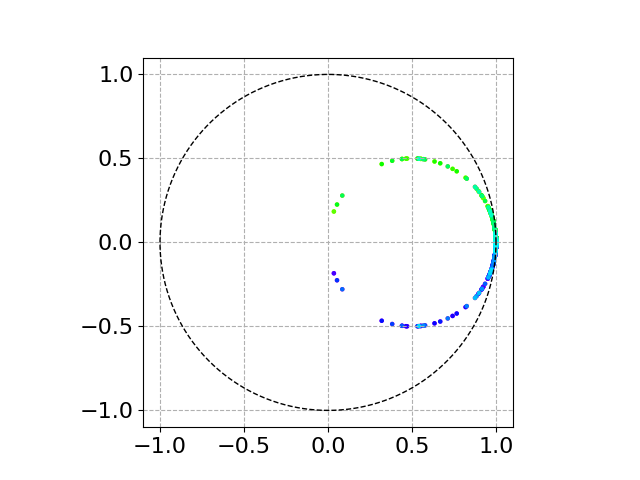}}}\hspace{5pt}
\subfloat[$\alpha=-10^{-6},\,\g=1$]{%
\resizebox*{5cm}{!}{\includegraphics{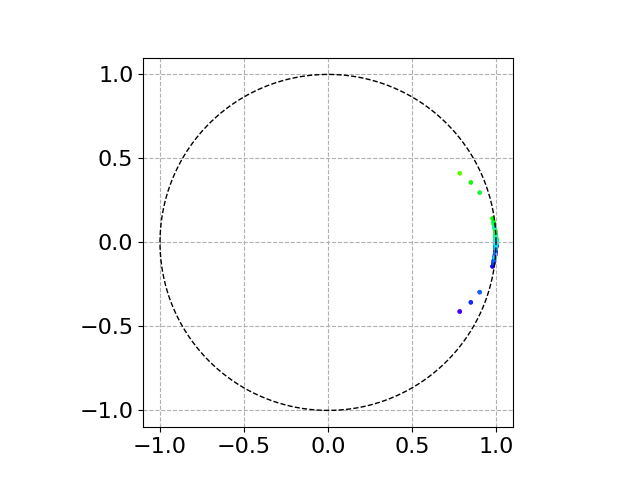}}}\hspace{5pt}

\subfloat[$\alpha=0,\;\g=1$]{%
\resizebox*{5cm}{!}{\includegraphics{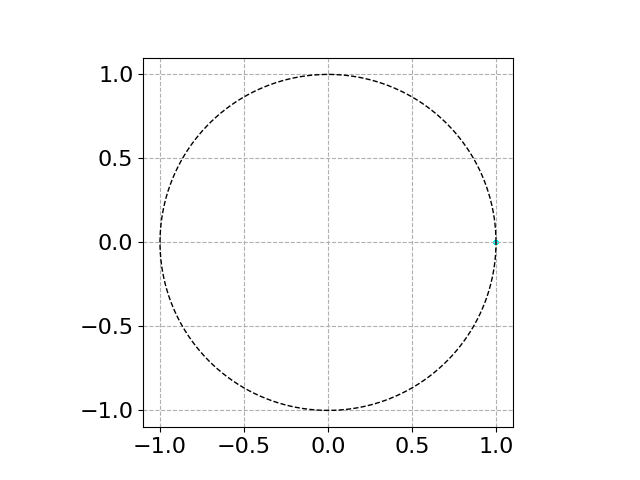}}}\hspace{5pt}

\subfloat[$\alpha=10^{-6},\;\g=1$]{%
\resizebox*{5cm}{!}{\includegraphics{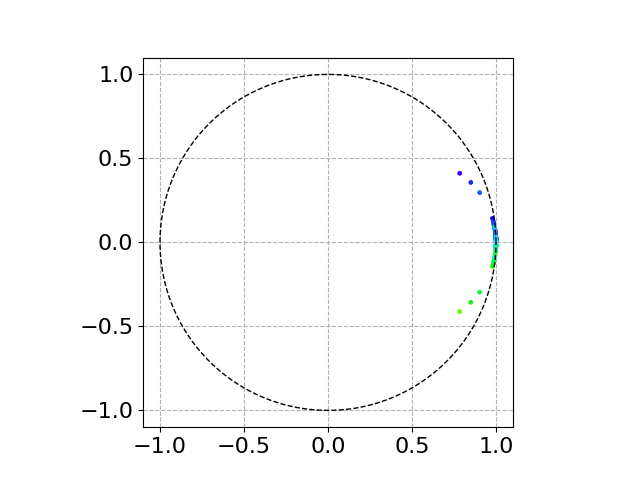}}}\hspace{5pt}
\subfloat[$\alpha=10^{-5},\;\g=1$]{%
\resizebox*{5cm}{!}{\includegraphics{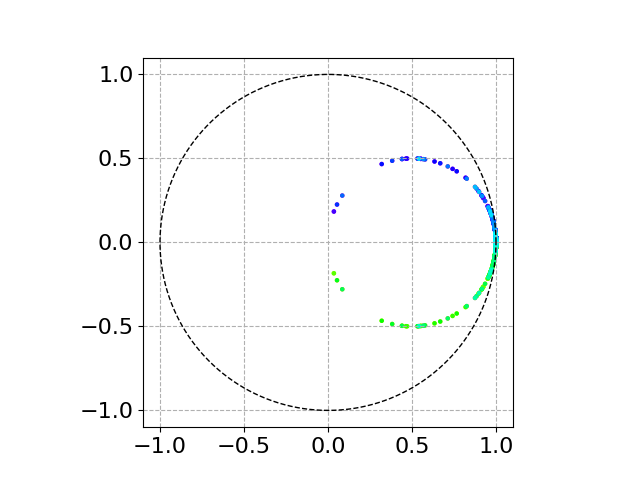}}}\hspace{5pt}
\subfloat[$\alpha=10^{-3},\;\g=1$]{%
\resizebox*{5cm}{!}{\includegraphics{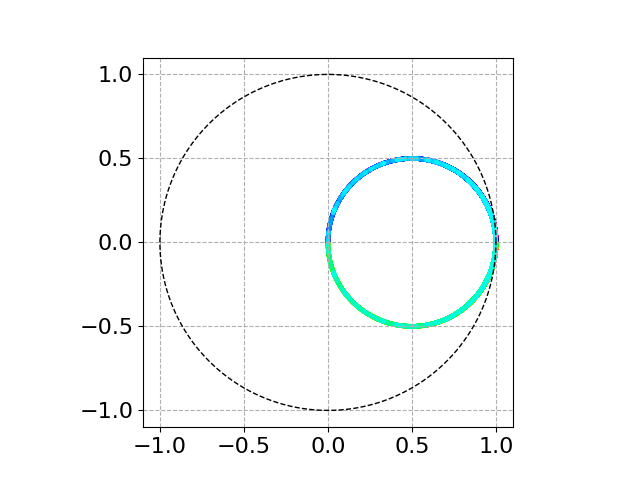}}}\hspace{5pt}

\subfloat[$\alpha=0.1,\;\g=1$]{%
\resizebox*{5cm}{!}{\includegraphics{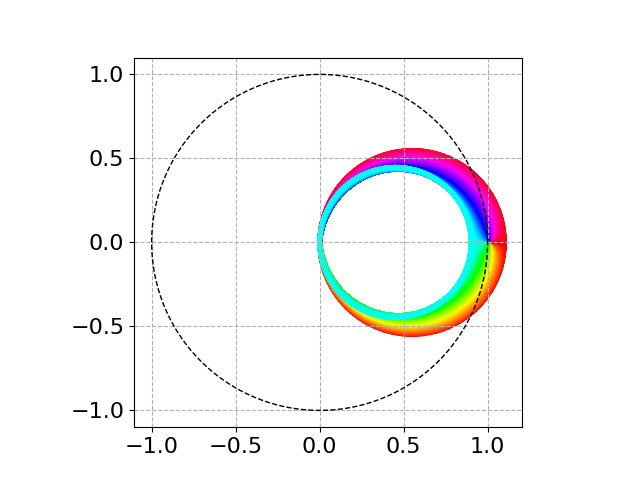}}}\hspace{5pt}
\subfloat[$\alpha=0.3,\;\g=1$]{%
\resizebox*{5cm}{!}{\includegraphics{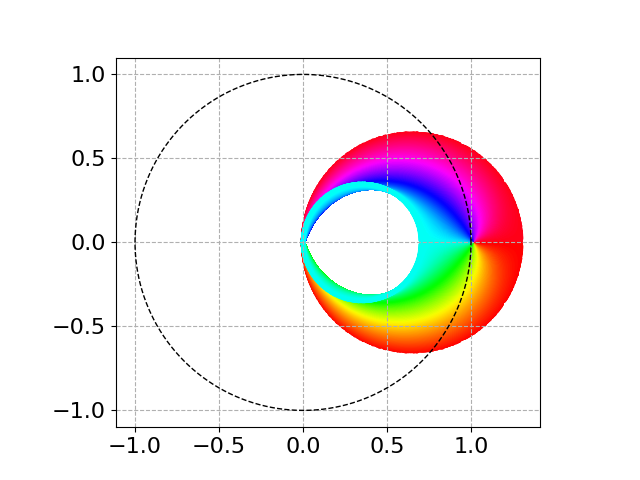}}}\hspace{5pt}
\subfloat[$\alpha=0.7,\;\g=1$]{%
\resizebox*{5cm}{!}{\includegraphics{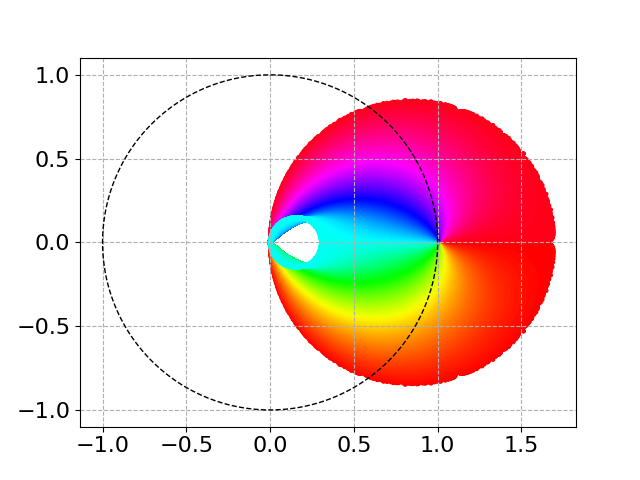}}}\hspace{5pt}

\caption{Comparison of Berezin ranges of $C_{\phi_\alpha}$ in Hardy space for different values of $\alpha$ when $\alpha$ varies along the real axis.}
\label{fig:comparison_g=1_a=real}
\end{figure}


\newpage
A brief summary of the complex symmetric structure of $D_{\mfn,\psi,\varphi}$ are described in {Table} 1.
\begin{table}[h]
    \centering
    \begin{tabular}{ |p{1.60cm}|p{6.00cm}|p{7.500cm}|}          
        \hline
        \textbf{Hilbert spaces} & \begin{center}
            \textbf{Unit disc} \( \mathbb{D} \)
        \end{center}   & \begin{center}
            \textbf{Unit polydisc} \( \mathbb{D}^d \)
        \end{center} \\
        \hline        
        \textcolor{purple}{Hardy space $\mathcal{H}^2(\mathbb{D})$} &$\psi(z)=\frac{1}{n!}az^{n}\big(1-z\varphi(0)\big)^{-n-1},$       $\varphi(z)=\varphi(0)+\frac{z\varphi^{\prime}(0)}{1-z\varphi(0)},$~~ 
where $a=\psi^n(0)$.   & $\psi(z)=a\prod_{j=1}^dz_j^{n_j}\big(1-z_j\varphi_j(0)\big)^{-n_j-1},$       $\varphi_t(z_t)=\varphi_t(0)+\frac{z_t\varphi_t^{\prime}(0)}{1-z_t\varphi_t(0)},~ \mbox{for}~ t=1, \dots, d,$ 
where $a=\partial^n\psi(0, \dots, 0)$.  \\
         \hline
        \textcolor{purple}{Bergman space $\mathcal{L}_a^2(\mathbb{D})$} &$\psi(z)=\frac{1}{n!}az^{n}\big(1-z\varphi(0)\big)^{-n-2},$       $\varphi(z)=\varphi(0)+\frac{z\varphi^{\prime}(0)}{1-z\varphi(0)},$~~~~  where $a=\psi^n(0)$.  &$\psi(z)=a\prod_{j=1}^dz_j^{n_j}\big(1-z_j\varphi_j(0)\big)^{-n_j-2},$       $\varphi_t(z_t)=\varphi_t(0)+\frac{z_t\varphi_t^{\prime}(0)}{1-z_t\varphi_t(0)},~ \mbox{for}~ t=1, \dots, d,$ 
where $a=\partial^n\psi(0, \dots, 0)$.
\\
         \hline
        \textcolor{purple}{Weighted Hardy space $\mathcal{H}_\gamma(\mathbb{D})$} &$ \psi(z)=az^{n}\big(1-z\varphi(0)\big)^{-\gamma-n},$ 
        $\varphi(z)=\varphi(0)+\frac{z\varphi^{\prime}(0)}{1-z\varphi(0)},$ 
        where $a=\frac{\psi^n(0)}{n!}$, $\gamma\in \mathbb{N}$.   & $ \psi(z)=a\prod_{j=1}^dz_j^{n_j}\big(1-z_j\varphi_j(0)\big)^{-\gamma-n_j},$ 
        $\varphi_t(z_t)=\varphi_t(0)+\frac{z_t\varphi_t^{\prime}(0)}{1-z_t\varphi_t(0)},~ \mbox{for}~ t=1, \dots, d,$ 
        where $a=\partial^n\psi(0, \dots, 0)$, $\gamma\in \mathbb{N}$.
 \\
        \hline
    \end{tabular}
    \caption{Comparison of complex symmetric structure of composition-differentiation operator with respect to
the conjugation $\mathcal{J}$ acting over the unit disc $\mathbb{D}$ and the polydisc $\mathbb{D}^d$.}
\end{table}

\section{Concluding remarks and future directions}

We conclude by remarking that the complex symmetric structure and self-adjoint property of generalized composition-differentiation operators over the unit polydisc and the unit disc with respect to different conjugations are investigated in this article. By employing similar analysis to different special classes of operators, it is possible to further study the complex symmetric structure of operators on any reproducing kernel Hilbert space (RKHS) of holomorphic functions. Moreover, one can explore the convexity property in other classes of operators with suitable reproducing kernel Hilbert spaces.

\vspace{.3cm}
	\noindent
	{\small {\bf Acknowledgments.}\\
	 The second author is supported by the Institute Postdoctoral Fellowship, IIT Bhubaneswar (F. 15-12/2024-Acad/SBS5-PDF-01).
	}

\subsection*{Declarations}
\begin{itemize}
\item {\bf{Data availability statements}}: Data sharing not applicable to this article as no datasets were generated or analysed during the current study.
\item {\bf{Competing interests}}: The authors have no relevant financial or non-financial interests to disclose.
\item {\bf{Funding}}: Not applicable.
\item {\bf{Authors' contributions}}: Authors declare they have contributed equally to this paper. Both the authors have read and approved this version.
\end{itemize}




\end{document}